\documentclass{eptcs}

\usepackage[T1]{fontenc}
\usepackage[utf8]{inputenc}
\usepackage{amsmath}
\usepackage{cases}
\usepackage{amsthm}
\usepackage{amsfonts}
\usepackage{bbm}
\usepackage{mathtools}
\usepackage{tikz}
\usepackage{pgf}
\usepackage{graphicx}
\usepackage{caption}
\usepackage{subcaption}
\usepackage{multicol}
\usepackage{todonotes}
\usepackage{float}
\usepackage{url}
\usepackage{apxproof}

\usetikzlibrary{arrows,automata}
\usetikzlibrary{positioning,calc,decorations.pathmorphing,decorations.pathreplacing,patterns}

\newtheorem{definition}{Definition}
\newtheoremrep{lemma}[definition]{Lemma}
\newtheorem{proposition}[definition]{Proposition}
\newtheorem{theorem}[definition]{Theorem}
\newtheorem{corollary}[definition]{Corollary}
\newtheorem{conjecture}[definition]{Conjecture}

\usepackage{todonotes}

\usepackage{tikz}
\usepackage{ifthen}
\usetikzlibrary{calc,patterns,snakes,arrows,decorations,decorations.markings,shapes.misc}

\pgfdeclaredecoration{dynamic rounded corners}{initial}{
    \state{initial}[width=\pgfdecoratedinputsegmentlength,
        next state=middle,
        persistent postcomputation=\pgfmathsetmacro\previousroundedendlength{min(\pgfdecorationsegmentlength,\pgfdecoratedinputsegmentlength)}
    ]{}
    \state{middle}[width=\pgfdecoratedinputsegmentlength,next state=middle,
      persistent precomputation={
         \pgfmathsetmacro\roundedstartlength{min(\previousroundedendlength,\pgfdecoratedinputsegmentlength/2)}
        },
        persistent postcomputation=\pgfmathsetmacro\previousroundedendlength{min(\pgfdecorationsegmentlength,\pgfdecoratedinputsegmentlength/2)}
    ]{
        \pgfsetcornersarced{\pgfpoint{\roundedstartlength}{\roundedstartlength}}
        \pgfpathlineto{\pgfpoint{0pt}{0pt}}
    }
    \state{final}[if input segment is closepath={\pgfpathclose}]
    {
        \pgfpathlineto{\pgfpointdecoratedpathlast}
    }
}

\pgfdeclaredecoration{dynamic rounded corners final}{initial}{
    \state{initial}[width=\pgfdecoratedinputsegmentlength,
        next state=middle,
        persistent postcomputation=\pgfmathsetmacro\previousroundedendlength{min(\pgfdecorationsegmentlength,\pgfdecoratedinputsegmentlength)}
    ]{}
    \state{middle}[width=\pgfdecoratedinputsegmentlength,next state=middle,
      persistent precomputation={
        \pgfmathapproxequalto{\pgfdecoratedinputsegmentlength}{\pgfdecoratedremainingdistance}
        \ifpgfmathcomparison
               \pgfmathsetmacro\roundedstartlength{min(\previousroundedendlength,\pgfdecoratedinputsegmentlength)}
        \else
               \pgfmathsetmacro\roundedstartlength{min(\previousroundedendlength,\pgfdecoratedinputsegmentlength/2)}
        \fi
        },
        persistent postcomputation=\pgfmathsetmacro\previousroundedendlength{min(\pgfdecorationsegmentlength,\pgfdecoratedinputsegmentlength/2)}
    ]{
        \pgfsetcornersarced{\pgfpoint{\roundedstartlength}{\roundedstartlength}}
        \pgfpathlineto{\pgfpoint{0pt}{0pt}}
    }
    \state{final}[if input segment is closepath={\pgfpathclose}]
    {
        \pgfpathlineto{\pgfpointdecoratedpathlast}
    }
}

\tikzset{
    dynamic rounded corners/.style={
        decorate,
        decoration={
            dynamic rounded corners,
            segment length=#1
        }
    }
}

\tikzset{
    dynamic rounded corners final/.style={
        decorate,
        decoration={
            dynamic rounded corners final,
            segment length=#1
        }
    }
}

\def\defaultfacecolor{green}
\def\defaultroundedness{3pt}
\def\defaultslicedistance{1}
\def\defaultbraiddistance{0.1}

\tikzset{cross/.style={cross out, draw, 
         minimum size=2*(#1-\pgflinewidth), 
         inner sep=0pt, outer sep=0pt}}
\tikzstyle{edge}=[]
\tikzstyle{bnode}=[draw,black,circle,fill=black,inner sep=1pt]
\tikzstyle{bplace}=[draw,blue,thick,cross,inner sep=2pt]
\tikzstyle{yspot}=[draw,\defaultfacecolor,circle,fill=\defaultfacecolor,inner sep=1pt]
\tikzstyle{eface}=[draw,\defaultfacecolor]
\tikzstyle{rnode}=[draw,red,circle,fill=red,inner sep=1pt]
\tikzstyle{wnode}=[draw,white,circle,fill=white,inner sep=1.5pt]
\tikzstyle{strike}=[edge,draw,black,circle,inner sep=1.5pt,strike out]
\tikzstyle{bstrike}=[edge,draw,black,circle,inner sep=1.5pt,strike out,rotate=90]
\tikzstyle{dedge}=[postaction={nomorepostaction,decorate,
                    decoration={markings,mark=at position 0.5 with {\arrow{>}}}
}]
\tikzstyle{blanknode}=[]
\tikzstyle{domainlabel}=[circle,fill=white,inner sep=1pt]

\makeatletter
\tikzset{nomorepostaction/.code={\let\tikz@postactions\pgfutil@empty}}
\makeatother

\newcommand\storeface[2]{\expandafter\xdef\csname faceid#1\endcsname{#2}}
\newcommand\getface[1]{\csname faceid#1\endcsname}
\newcounter{nextfaceid}
\newcounter{maxfaceseen}

\newcommand\storefinalface[2]{\expandafter\xdef\csname finalfaceid#1\endcsname{#2}}
\newcommand\getfinalface[1]{\csname finalfaceid#1\endcsname}

\newcommand\storefacecolor[2]{\expandafter\xdef\csname facecolor#1\endcsname{#2}}
\newcommand\getfacecolor[1]{\csname facecolor#1\endcsname}

\newcommand\storeminslice[2]{\expandafter\xdef\csname minfaceslice#1\endcsname{#2}}
\newcommand\getminslice[1]{\csname minfaceslice#1\endcsname}

\newcommand\storeedgeid[2]{\expandafter\xdef\csname edgeid#1\endcsname{#2}}
\newcommand\getedgeid[1]{\csname edgeid#1\endcsname}
\newcounter{nextedgeid}

\newcommand\storeedgepath[2]{\expandafter\xdef\csname edgepath#1\endcsname{#2}}
\newcommand\getedgepath[1]{\csname edgepath#1\endcsname}
\newcommand\expandedgepath[2]{\storeedgepath{\getedgeid{#1}}{\getedgepath{\getedgeid{#1}} #2}}

\newcommand\storeedgecolor[2]{\expandafter\xdef\csname edgecolor#1\endcsname{#2}}
\newcommand\getedgecolor[1]{\csname edgecolor#1\endcsname}

\newcounter{nextdeferrednodeid}
\newcommand\storedeferrednodeid[2]{\expandafter\xdef\csname deferrednodeid#1\endcsname{#2}}
\newcommand\storedeferrednodestyle[2]{\expandafter\xdef\csname deferrednodestyle#1\endcsname{#2}}
\newcommand\getdeferrednodeid[1]{\csname deferrednodeid#1\endcsname}
\newcommand\getdeferrednodestyle[1]{\csname deferrednodestyle#1\endcsname}

\newcommand\mergefaces[2]{
   \pgfmathtruncatemacro\valfa{\getfinalface{#1}}
   \pgfmathtruncatemacro\valfb{\getfinalface{#2}}
   \pgfmathtruncatemacro\newfaceid{min(\valfa, \valfb)}
   \pgfmathtruncatemacro\latestfaceid{\themaxfaceseen - 1}
   \foreach \x in {0,...,\latestfaceid}{
       \ifthenelse{\getfinalface{\x}=\valfa \OR \getfinalface{\x}=\valfb}{
          \storefinalface{\x}{\newfaceid}
       }{}
   }     
}

\def\drawfaces{0}

\def\capscups{0}
\newcommand\enablecapscups[0]{
  \def\capscups{1};
  \def\slicedistance{1.75}
}
\newcommand\disablecapscups[0]{
  \def\capscups{0}
  \ifthenelse{\multislice=0}{
    \def\slicedistance{1}
  }{}
}

\def\writefaceids{0}

\def\drawplaces{0}

\newcommand\startdiagram[1]{
   \wlog{-- Start Diagram --}
   \pgfmathtruncatemacro\maxstrandidx{#1 - 1}
   \pgfmathsetmacro\centeringoffset{-0.5 * #1}
   \pgfmathtruncatemacro\diagramlevel{0}
   \pgfmathsetmacro\verticalpos{0}

   \foreach \x in {0,...,\maxstrandidx}{
      \storeface{\x}{\x}
      \storefinalface{\x}{\x}
      \storefacecolor{\x}{\defaultfacecolor}
      \storeminslice{\x}{0}
   }
   \pgfmathtruncatemacro\tmpnextfaceid{\maxstrandidx + 1}
   \setcounter{nextfaceid}{\tmpnextfaceid}
   \setcounter{maxfaceseen}{\tmpnextfaceid}

   \def\slicedistance{\defaultslicedistance}
   \def\multislice{0}
   \def\multislicelength{0}
   \def\multislicebeforeoffset{0}
   \def\multisliceafteroffset{0}
   \def\multislicevertexid{1}

   \def\braidmode{0}
   \def\braiddistance{\defaultbraiddistance}

   \setcounter{nextedgeid}{0}

   \setcounter{nextdeferrednodeid}{0}
}

\newcommand\drawinitialstrands[1]{
   \pgfmathtruncatemacro\maxstrandidx{#1 - 1}
   \pgfmathsetmacro\centeringoffset{-0.5 * #1}
   \pgfmathtruncatemacro\diagramlevel{0}

   \foreach \x in {0,...,\maxstrandidx}{
      \storeface{\x}{\x}
   }

  \begin{scope}[xshift=\centeringoffset cm]
   \ifthenelse{\maxstrandidx = 0}{}{
   \foreach \x in {1,...,\maxstrandidx}{

     \node at ($(\x, .7)$) (input_\x) {};

     \storeedgeid{\x}{\thenextedgeid}
     \storeedgepath{\thenextedgeid}{($(\x + \centeringoffset, .5)$) -- ($(\x+\centeringoffset,0)$)};
     \storeedgecolor{\thenextedgeid}{black}
     \stepcounter{nextedgeid}
   }
   }
   \end{scope}

    \foreach \x in {0,...,\maxstrandidx} {
        \node at ($(\x+\centeringoffset + .5, \verticalpos)$) (spot_\diagramlevel_\x) {};
    }
    \foreach \x in {1,...,\maxstrandidx} {
      \node at ($(\x+\centeringoffset, \verticalpos)$) (place_\diagramlevel_\x) {};
    }

   \pgfmathtruncatemacro\tmpnextfaceid{\maxstrandidx + 1}
   \setcounter{nextfaceid}{\tmpnextfaceid}

   \pgfmathtruncatemacro\latestfaceid{\themaxfaceseen - 1}
}

\newcommand\finishdiagram[0]{
  \wlog{finishing diagram with \maxstrandidx wires}
   \ifthenelse{\maxstrandidx = 0}{}{
     \foreach \x in {1,...,\maxstrandidx}{
       \node at ($(\x + \centeringoffset, \verticalpos - .7)$) (output_\x) {};
       \expandedgepath{\x}{ -- +(0,-.5)};
   }
   }

   \pgfmathtruncatemacro\maxedgeid{\thenextedgeid - 1}
   \ifthenelse{\thenextedgeid = 0}{}{
     \foreach \eid in {0,...,\maxedgeid}{
     
       \draw[edge,\getedgecolor{\eid},dynamic rounded corners final=\defaultroundedness] \getedgepath{\eid} ;

   }
   }

   \pgfmathtruncatemacro\maxdeferredvertex{\thenextdeferrednodeid - 1}
   \ifthenelse{\thenextdeferrednodeid > 0}{
     \foreach \x in {0,...,\maxdeferredvertex} {
       \node[\getdeferrednodestyle{\x}] at (v\getdeferrednodeid{\x}) {};
     }
   }{}
}

\newcommand\updatefinalfaceids[3]{
    \def\wb{#1}
    \def\inputs{#2}
    \def\outputs{#3}
    \pgfmathtruncatemacro\nextdiaglevel{\diagramlevel+1}
    \pgfmathsetmacro\nextverticalpos{\verticalpos - \slicedistance}
    \pgfmathtruncatemacro\horizoffset{\outputs - \inputs}
    \pgfmathtruncatemacro\bottomrightcorner{\wb + \outputs + 1}
    \pgfmathtruncatemacro\innertoprightcorner{\wb + \inputs}

    \ifthenelse{\outputs=0 \AND \inputs>0}{
      \mergefaces{\getface{\wb}}{\getface{\innertoprightcorner}}
    }{}

   \ifthenelse{\horizoffset > 0}{
       \foreach \x in {\maxstrandidx,...,\innertoprightcorner}{
          \pgfmathtruncatemacro\offsetidx{\x + \horizoffset}
          \storeface{\offsetidx}{\getface{\x}}
       }
   }{
       \foreach \x in {\innertoprightcorner,...,\maxstrandidx}{
          \pgfmathtruncatemacro\offsetidx{\x + \horizoffset}
          \storeface{\offsetidx}{\getface{\x}}
       }
   }

   \pgfmathtruncatemacro\firstnewface{\wb + 1}
   \pgfmathtruncatemacro\finalnewface{\wb + \outputs - 1}
   \ifthenelse{\outputs > 1}{
     \foreach \x in {\firstnewface,...,\finalnewface}{
       \storeface{\x}{\thenextfaceid}

       \storeminslice{\thenextfaceid}{\nextdiaglevel}

       \ifthenelse{\themaxfaceseen > \thenextfaceid}{}{
         \storefinalface{\thenextfaceid}{\thenextfaceid}
         \storefacecolor{\thenextfaceid}{eface}
         \stepcounter{maxfaceseen}
       }
       \stepcounter{nextfaceid}

     }
   }{}
   \pgfmathtruncatemacro\latestfaceid{\themaxfaceseen - 1}
   
}
  
\newcommand\scanslice[3]{

    \pgfmathtruncatemacro\maxstrandidx{\maxstrandidx + \horizoffset}
    \pgfmathtruncatemacro\diagramlevel{\diagramlevel+1}
    \pgfmathsetmacro\verticalpos{\verticalpos - \slicedistance}

}

\newcommand\diagslicenovertex[3]{
    \def\wb{#1}
    \def\inputs{#2}
    \def\outputs{#3}
    \wlog{DIAGSLICE #1 #2 #3}
    \pgfmathtruncatemacro\nextdiaglevel{\diagramlevel+1}
    \pgfmathsetmacro\nextverticalpos{\verticalpos - \slicedistance}
    \pgfmathtruncatemacro\horizoffset{\outputs - \inputs}
    \ifthenelse{\multislice = 0}{
      \pgfmathsetmacro\nextoffset{\centeringoffset - 0.5*\horizoffset}
    }{
      \pgfmathsetmacro\nextoffset{-0.5 - 0.5*\multislicelength}
    }
    \pgfmathtruncatemacro\wbtop{\multislicebeforeoffset + \wb}
    \pgfmathtruncatemacro\wbbot{\multisliceafteroffset + \wb}
    \pgfmathtruncatemacro\toprightcorner{\wbtop + \inputs + 1}
    \pgfmathtruncatemacro\bottomrightcorner{\wbbot + \outputs + 1}
    \ifthenelse{\multislice = 0}{
      \pgfmathsetmacro\vertexpos{0.5* (\centeringoffset + \wb + 0.5 + 0.5*\inputs) + 0.5*(\nextoffset + \wb + 0.5 + 0.5*\outputs)}
    }{
      \ifthenelse{\inputs=0}{
         \pgfmathsetmacro\vertexpos{\nextoffset + \wbbot + 0.5 + 0.5*\outputs}
      }{
        \ifthenelse{\outputs=0}{
          \pgfmathsetmacro\vertexpos{\centeringoffset + \wbtop + 0.5 + 0.5*\inputs}
        }{
          \pgfmathsetmacro\vertexpos{0.5* (\centeringoffset + \wbtop + 0.5 + 0.5*\inputs) + 0.5*(\nextoffset + \wbbot + 0.5 + 0.5*\outputs)}
        }
      }
    }

    \wlog{Positioning vertex at vertical pos: between \verticalpos  and \nextverticalpos}
    \wlog{and horizontal pos: \vertexpos}
    \node (curvertex) at ($(\vertexpos, 0.5*\verticalpos + 0.5*\nextverticalpos)$) {};
    \node at (curvertex) (v\diagramlevel) {};
    \node at (curvertex) (v\diagramlevel-\multislicevertexid) {};

    \ifthenelse{\wb=0}{}{
       \pgfmathtruncatemacro\msbop{\multisliceafteroffset + 1}
        \foreach \x in {\msbop,...,\wbbot} {
          \expandedgepath{\x}{-- ($(\x + \nextoffset,\nextverticalpos)$)};
        }
    }
       
    \ifthenelse{\inputs=0}{}{
      \foreach \x in {1,...,\inputs} {
         \pgfmathtruncatemacro\edgepos{\wbbot+\x}

         \ifthenelse{\capscups=0 \OR \( \not \( \inputs=2 \) \OR \not \( \outputs=0 \) \)}{
         }{
           \expandedgepath{\edgepos}{-- ($(\x + \wbtop + \centeringoffset, .5*\verticalpos + .5*\nextverticalpos)$)};
         }

         \ifthenelse{\braidmode=1 \AND \x=2}{
           \expandedgepath{\edgepos}{-- ($(\vertexpos + \braiddistance, .5*\verticalpos + .5*\nextverticalpos + \braiddistance)$)};
         }{
           \ifthenelse{\braidmode=2 \AND \x=1}{
             \expandedgepath{\edgepos}{-- ($(\vertexpos - \braiddistance, .5*\verticalpos + .5*\nextverticalpos + \braiddistance)$)};
           }{
             \expandedgepath{\edgepos}{-- ($(\vertexpos, .5*\verticalpos + .5*\nextverticalpos)$)};
             
             \ifthenelse{\braidmode=0}{}{
               \xdef\overbraidedgeid{\getedgeid{\edgepos}}
             }
           }
         }
       }
    }

    \wlog{shifting edge ids?}
    \ifthenelse{\toprightcorner > \maxstrandidx}{}{
      \wlog{yes there are remaining wires}
      \ifthenelse{\inputs > \outputs}{
           \pgfmathtruncatemacro\maxstrandidxoffset{\bottomrightcorner + \maxstrandidx - \toprightcorner}
           \wlog{previously shifting edges from \toprightcorner  to \maxstrandidx by \horizoffset}
           \wlog{now from \bottomrightcorner  to \maxstrandidxoffset by \horizoffset}
           \foreach \x in {\bottomrightcorner,...,\maxstrandidxoffset}{
              \pgfmathtruncatemacro\edgepos{\x - \horizoffset}
              \storeedgeid{\x}{\getedgeid{\edgepos}}
           }
         }{}
      \ifthenelse{\outputs > \inputs}{
        \wlog{shifting edges from \maxstrandidx  down to \toprightcorner by \horizoffset}
           \foreach \x in {\maxstrandidx,...,\toprightcorner}{
              \pgfmathtruncatemacro\edgepos{\x + \horizoffset}
              \storeedgeid{\edgepos}{\getedgeid{\x}}
           }           
         }{}
    }

    \ifthenelse{\outputs=0}{}{
      \foreach \x in {1,...,\outputs} {
        \pgfmathtruncatemacro\edgepos{\wbbot + \x}
        \storeedgeid{\edgepos}{\thenextedgeid}

        \ifthenelse{\braidmode=1 \AND \x=1}{
          \storeedgepath{\thenextedgeid}{($(\vertexpos - \braiddistance, .5*\verticalpos + .5*\nextverticalpos - \braiddistance)$)};
        }{
          \ifthenelse{\braidmode=2 \AND \x=2}{
            \storeedgepath{\thenextedgeid}{($(\vertexpos + \braiddistance, .5*\verticalpos + .5*\nextverticalpos - \braiddistance)$)};
          }{
            \ifthenelse{\braidmode=0}{
              \storeedgepath{\thenextedgeid}{($(\vertexpos, .5*\verticalpos + .5*\nextverticalpos)$)};
            }{
              \storeedgeid{\edgepos}{\overbraidedgeid}
            }
          }
        }

        \ifthenelse{\capscups=0 \OR \not \( \( \outputs=2 \) \OR \not \( \inputs=0 \) \)}{

        }{
           \expandedgepath{\edgepos}{-- ($(\edgepos + \nextoffset, .5*\verticalpos + .5*\nextverticalpos)$)};
        }
        
        \expandedgepath{\edgepos}{ -- ($(\edgepos + \nextoffset,\nextverticalpos)$)};

        \ifthenelse{\( \braidmode=0 \) \OR \( \x=\braidmode \) }{
           \storeedgecolor{\thenextedgeid}{black}
           \stepcounter{nextedgeid}
        }{}
       }
    }

    \ifthenelse{\multislice=0}{
      \ifthenelse{\toprightcorner > \maxstrandidx}{}{
        \pgfmathtruncatemacro\nextmaxstrandidx{\maxstrandidx + \horizoffset}
        \foreach \x in {\bottomrightcorner,...,\nextmaxstrandidx} {
          \expandedgepath{\x}{-- ($(\x + \nextoffset, \nextverticalpos)$)};
        }
      }
    }{
      \pgfmathtruncatemacro\multislicebeforeoffset{\wbtop + \inputs}
      \pgfmathtruncatemacro\multisliceafteroffset{\wbbot + \outputs}
      \pgfmathtruncatemacro\multislicevertexid{\multislicevertexid + 1}
    }

    \ifthenelse{\drawfaces=0}{}{
         \ifthenelse{\inputs=0 \AND \outputs=0}{
            \pgfmathtruncatemacro\innertopleftcorner{\wb - 1}
            \pgfmathtruncatemacro\innertoprightcorner{\toprightcorner}

            \draw[eface] ($(\wb + \centeringoffset + .5, \verticalpos)$) edge[bend left=40] ($(\wb + \nextoffset + .5, \nextverticalpos)$);
            \draw[eface] ($(\wb + \centeringoffset + .5, \verticalpos)$) edge[bend right=40] ($(\wb + \nextoffset + .5, \nextverticalpos)$);
        }{
            \pgfmathtruncatemacro\innertopleftcorner{\wb}
            \pgfmathtruncatemacro\innertoprightcorner{\toprightcorner - 1}
        }

        \foreach \x in {0,...,\innertopleftcorner} {
                 \draw[eface,\getfacecolor{\getfinalface{\getface{\x}}}] ($(\x + \centeringoffset + .5, \verticalpos)$) -- ($(\x + \nextoffset + .5, \nextverticalpos)$);
        }
         
         \foreach \x in {\innertoprightcorner,...,\maxstrandidx} {
            \draw[eface,\getfacecolor{\getfinalface{\getface{\x}}}] ($(\x + \centeringoffset + .5, \verticalpos)$) -- ($(\x + \horizoffset + \nextoffset + .5, \nextverticalpos)$);
         }

    }

    \ifthenelse{\multislice=0}{
      \pgfmathtruncatemacro\maxstrandidx{\maxstrandidx + \horizoffset}
      \pgfmathtruncatemacro\diagramlevel{\diagramlevel+1}
      \pgfmathsetmacro\verticalpos{\verticalpos - \slicedistance}

      \pgfmathsetmacro\centeringoffset{\nextoffset}
    }{}

    \foreach \x in {0,...,\maxstrandidx} {
        \node at ($(\x+\centeringoffset + .5, \verticalpos)$) (spot_\diagramlevel_\x) {};
    }
    \foreach \x in {1,...,\maxstrandidx} {
      \node at ($(\x+\centeringoffset, \verticalpos)$) (place_\diagramlevel_\x) {};
    }
    \ifthenelse{\drawfaces=0}{}{
       \placesandspots
    }
}

\newcommand\finishcurrentslice[0]{
    \ifthenelse{\toprightcorner > \maxstrandidx}{}{
      \pgfmathtruncatemacro\nextmaxstrandidx{\maxstrandidx + \horizoffset}
        \foreach \x in {\bottomrightcorner,...,\nextmaxstrandidx} {
          \expandedgepath{\x}{-- ($(\x + \nextoffset, \nextverticalpos)$)};
        }
    }
        
    \pgfmathtruncatemacro\maxstrandidx{\multislicelength}
    \pgfmathtruncatemacro\diagramlevel{\diagramlevel+1}
    \pgfmathsetmacro\verticalpos{\verticalpos - \slicedistance}
    \def\slicedistance{1}

    \pgfmathsetmacro\centeringoffset{\nextoffset}
}

\newcommand\diagslice[3]{
  \diagslicenovertex{#1}{#2}{#3};
  \node[bnode] at (curvertex) {};
}

\newcommand\pbraidslice[1]{
  \def\braidmode{1}
  \diagslicenovertex{#1}{2}{2};
  \def\braidmode{0}
}

\newcommand\nbraidslice[1]{
  \def\braidmode{2}
  \diagslicenovertex{#1}{2}{2};
  \def\braidmode{0}
}

\newcommand\cupslice[1]{
  \enablecapscups
  \diagslicenovertex{#1}{2}{0}
  \disablecapscups
}

\newcommand\capslice[1]{
  \enablecapscups
  \diagslicenovertex{#1}{0}{2}
  \disablecapscups
}

\newcommand\cupright[1]{
  \cupslice{#1}
  \node[xshift=2pt,rotate=-90,baseline=0] at (curvertex) {\tikz\draw[very thick,->] (0,0);};
}

\newcommand\cupleft[1]{
  \cupslice{#1}
  \node[xshift=-2pt,rotate=90,baseline=0] at (curvertex) {\tikz\draw[very thick,->] (0,0);};
}

\newcommand\capright[1]{
  \capslice{#1}
  \node[xshift=2pt,rotate=-90,baseline=0] at (curvertex) {\tikz\draw[very thick,->] (0,0);};
}

\newcommand\capleft[1]{
  \capslice{#1}
  \node[xshift=-2pt,rotate=90,baseline=0] at (curvertex) {\tikz\draw[very thick,->] (0,0);};
}

\newcommand\placesandspots[0]{
    \foreach \x in {0,...,\maxstrandidx} {
        \node[yspot,\getfacecolor{\getfinalface{\getface{\x}}}] at (spot_\diagramlevel_\x) {};
        
       \ifthenelse{\writefaceids=1 \AND \diagramlevel=\getminslice{\getface{\x}} \AND \getfinalface{\getface{\x}}=\getface{\x}}{
            \node[node distance=.25cm,above of=spot_\diagramlevel_\x] {\small \getface{\x}};
        }{}
    }

    \ifthenelse{\maxstrandidx=0 \OR \drawplaces=0}{}{
       \foreach \x in {1,...,\maxstrandidx} {
          \node[bplace] at (place_\diagramlevel_\x) {};
       }
    }
}

\newenvironment{stringdiagram}[1]{
  \pgfmathtruncatemacro\nbinputsp{#1 + 1}
  \startdiagram{\nbinputsp}
  \ifthenelse{#1 > 0}{
    \drawinitialstrands{\nbinputsp}
  }{}
}{
  \finishdiagram
}

\newcommand\startslice[1]{
  \def\multislice{1}
  \def\multislicelength{#1}
  \def\multislicebeforeoffset{0}
  \def\multisliceafteroffset{0}
  \def\multislicevertexid{1}
}

\newcommand\finishslice[0]{
  \finishcurrentslice
  \def\multislice{0}
  \def\multislicebeforeoffset{0}
  \def\multisliceafteroffset{0}
  \def\multislicevertexid{1}
}

\newcommand{\cC}[0]{%
\mathcal{C}}

\newcommand{\citep}[1]{\cite{#1}}

\tikzstyle{edge}=[thick]
\tikzstyle{morphismlabel}=[node distance=0.3cm]

\DeclareMathOperator{\dom}{dom}
\DeclareMathOperator{\cod}{cod}
\def\ROTang{\mathbb{ROT}\mathbbm{ang}}
\def\CC{\mathbb{CC}}
\DeclareMathOperator{\ccc}{ccc}

\def\defaultroundedness{15pt}

\newcommand{\scapl}{\begin{tikzpicture}[baseline=-.45cm,scale=0.25] %
      \begin{stringdiagram}{0} %
        \capleft{0} %
      \end{stringdiagram} %
\end{tikzpicture}}
\newcommand{\scapr}{\begin{tikzpicture}[baseline=-.45cm,scale=0.25] %
      \begin{stringdiagram}{0} %
        \capright{0} %
      \end{stringdiagram} %
\end{tikzpicture}}
\newcommand{\scupl}{\begin{tikzpicture}[baseline=-.2cm,scale=0.25] %
      \begin{stringdiagram}{2} %
        \cupleft{0} %
      \end{stringdiagram} %
\end{tikzpicture}}
\newcommand{\scupr}{\begin{tikzpicture}[baseline=-.2cm,scale=0.25] %
      \begin{stringdiagram}{2} %
        \cupright{0} %
      \end{stringdiagram} %
\end{tikzpicture}}

\def\defaultbraiddistance{0.17}

\title{The Word Problem for Braided Monoidal Categories is Unknot-Hard}
\date{\today}
\author{Antonin Delpeuch
  \institute{Department of Computer Science, University of Oxford}
  \email{antonin.delpeuch@cs.ox.ac.uk}
  \and
  Jamie Vicary
  \institute{Computer Laboratory, University of Cambridge}
  \email{jamie.vicary@cl.cam.ac.uk}
}
\def\authorrunning{A. Delpeuch \& J. Vicary} 
\def\titlerunning{The Word Problem for Braided Monoidal Categories is Unknot-Hard}

\begin{document}

\maketitle

\begin{abstract}
  We show that the word problem for braided monoidal categories is at least
  as hard as the unknotting problem. As a corollary, so is the word problem for Gray categories. We conjecture that the word problem for Gray categories is decidable.
\end{abstract}

\section*{Introduction}

The \emph{word problem} for an algebraic structure is the decision problem which consists in determining whether two expressions denote the same elements of such a structure. Depending on the equational theory of the structure, this problem can be very simple or extremely difficult, and studying it from the lens of complexity or computability theory has proved insightful in many cases.

This work is the third episode of a series of articles studying the word problem for various sorts of categories, after monoidal categories~\citep{delpeuch2018normalization-1} and double categories~\citep{delpeuch2020word}. We turn here to braided monoidal categories. Unlike the previous episodes, we do not propose an algorithm deciding equality, but instead show that this word problem seems to be a difficult one. More precisely we show that it is at least as hard as the unknotting problem.

The unknotting problem consists in determining whether a knot can be untied and was first formulated by Dehn in 1910~\citep{Dehn1910}. The decidability of this problem remained open until Haken gave the first algorithm for it in 1961~\citep{haken1961theorie}. As of today, no polynomial time algorithm is known for it.

One reason why we are interested in braided monoidal categories is that they are a particularly natural sort of categories, in the sense that they arise as doubly degenerate Gray categories~\citep{gurski2011periodic}. Studying the word problem for them is therefore a first step towards understanding the word problem for weak higher categories, for which little is known to date.

This article starts off with a section giving some background on the various flavours of monoidal categories we will use, as well as a quick introduction to the unknotting problem. Then, we give a first reduction between the unknotting problem to the braided pivotal word problem, as a way of introducing tools which will be needed for the last section, where our main result is proved.

\subsection*{Acknowledgements}

The authors wish to thank the participants of the 2019 Postgraduate Conference in Category Theory and its Applications in Leicester and Makoto Ozawa for their feedback and help on this project. The first author is supported by an EPSRC scholarship.

\section{Background}

We assume familiarity with monoidal categories and their string diagrams~\citep{selinger2010survey}.
  
\subsection{Braided monoidal categories}

Braided monoidal categories were introduced by~\cite{joyal1986braided,joyal1993braided}. In this work,
we study the word problem for this algebraic structure. This is the decision problem where given
two expressions of morphisms in a free braided monoidal category, we need to determine whether or not they
represent the same morphism.

In what follows, all monoidal categories are strict. There exists weak
versions of the following definitions, and coherence theorems show
their equivalence with the strict definitions that we use here. See for instance
Theorem 4 in~\cite{joyal1986braided} for the case of braided monoidal categories.

\begin{definition}
  A \textbf{braided monoidal category} $\cC$ is a monoidal category $(\cC, \otimes, I)$
  equipped with a natural isomorphism $\sigma_{A,B} : A \otimes B \to B \otimes A$, satisfying
  the hexagon identities:
  
  $$\sigma_{A,B \otimes C} = (1_B \otimes \sigma_{A,C}) \circ (\sigma_{A,B} \otimes 1_C)$$
   $$\sigma_{A \otimes B, C} = (\sigma_{A,C} \otimes 1_B) \circ (1_A \otimes \sigma_{B,C})$$
\end{definition}

We use string diagrams for monoidal categories to represent morphisms
in braided monoidal categories. Figure~\ref{fig:braid} shows the
representation of the braid morphism and its
inverse. Figure~\ref{fig:hexagon} shows the representation of the
hexagon identities with this convention.

\begin{figure}
  \centering
  \begin{subfigure}{0.4\textwidth}
  \centering
  \begin{tikzpicture}[scale=0.5]
    \begin{stringdiagram}{2}
      \node at (-.5,1) {$A$};
      \node at (.5,1) {$B$};
      \node at (-.5,-2) {$B$};
      \node at (.5,-2) {$A$};
      \pbraidslice{0}
    \end{stringdiagram}
  \end{tikzpicture}
  \caption{String diagram for $\sigma_{A,B}$}
\end{subfigure}
\begin{subfigure}{0.4\textwidth}
  \centering
  \begin{tikzpicture}[scale=0.5]
    \begin{stringdiagram}{2}
      \node at (-.5,1) {$B$};
      \node at (.5,1) {$A$};
      \node at (-.5,-2) {$A$};
      \node at (.5,-2) {$B$};
      \nbraidslice{0}
    \end{stringdiagram}
  \end{tikzpicture}
  \caption{String diagram for $\sigma_{A,B}^{-1}$}
\end{subfigure}

  \caption{Representation of braid morphisms as string diagrams}
  \label{fig:braid}
\end{figure}
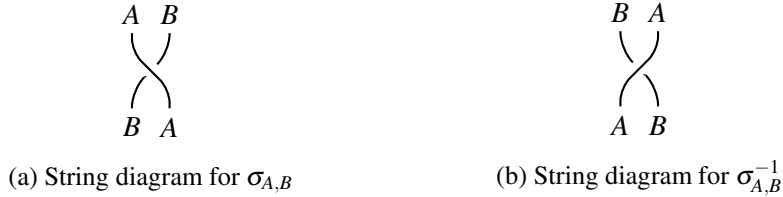

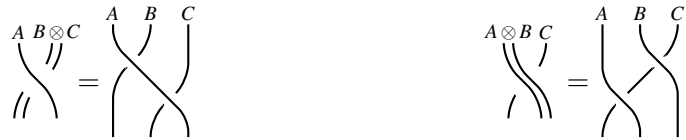
\begin{figure}
  \centering
  \begin{subfigure}{0.4\textwidth}
  \centering
  \begin{tikzpicture}[scale=.5]
    \begin{stringdiagram}{2}
      \diagslicenovertex{0}{2}{2}
      \storeedgecolor{0}{white}
      \storeedgecolor{3}{white}
      \storeedgecolor{1}{white}
      \storeedgecolor{2}{white}
    \end{stringdiagram}
    \foreach \xs in {0.125cm,-0.125cm} {
      \begin{scope}[xshift=\xs]
        \draw[edge,dynamic rounded corners final=\defaultroundedness] \getedgepath{1};
        \draw[edge,dynamic rounded corners final=\defaultroundedness] \getedgepath{2};
      \end{scope}
    }
    \node[fill=white,circle,inner sep=5pt] at (v0) {};
    \draw[edge,dynamic rounded corners final=\defaultroundedness] \getedgepath{0};
    \draw[edge,dynamic rounded corners final=\defaultroundedness] \getedgepath{3};
    \node[scale=.7] at (-.5,.8) {$A$};
    \node[scale=.7] at (.5,.8) {$B \otimes C$};

    \node at (1.35,-.6) {$=$};
    \begin{scope}[xshift=3cm,yshift=.5cm]
      \begin{stringdiagram}{3}
        \node[scale=.7] at (-1,.8) {$A$};
        \node[scale=.7] at (0,.8) {$B$};
        \node[scale=.7] at (1,.8) {$C$};
        \pbraidslice{0}
        \pbraidslice{1}
      \end{stringdiagram}
    \end{scope}
  \end{tikzpicture}
\end{subfigure}
\begin{subfigure}{0.4\textwidth}
  \centering
  \begin{tikzpicture}[scale=.5]
    \begin{stringdiagram}{2}
      \diagslicenovertex{0}{2}{2}
      \storeedgecolor{0}{white}
      \storeedgecolor{3}{white}
    \end{stringdiagram}
    \node[fill=white,circle,inner sep=5pt] at (v0) {};
    \foreach \xs in {0.125cm,-0.125cm} {
      \begin{scope}[xshift=\xs]
        \draw[edge,dynamic rounded corners final=\defaultroundedness] \getedgepath{0};
        \draw[edge,dynamic rounded corners final=\defaultroundedness] \getedgepath{3};
      \end{scope}
    }
    
    \node[scale=.7] at (-.5,.8) {$A \otimes B$};
    \node[scale=.7] at (.5,.8) {$C$};
    \node at (1.35,-.6) {$=$};
    \begin{scope}[xshift=3cm,yshift=.5cm]
      \begin{stringdiagram}{3}
        \node[scale=.7] at (-1,.8) {$A$};
        \node[scale=.7] at (0,.8) {$B$};
        \node[scale=.7] at (1,.8) {$C$};
        \pbraidslice{1}
        \pbraidslice{0}
      \end{stringdiagram}
    \end{scope}
  \end{tikzpicture}
\end{subfigure}

  \caption{The hexagon identities represented as string diagrams}
  \label{fig:hexagon}
\end{figure}

\begin{figure}
  \centering
  \begin{subfigure}{0.4\textwidth}
  \centering
    \begin{tikzpicture}[scale=.3]
    \begin{stringdiagram}{2}
      \pbraidslice{0}
      \nbraidslice{0}
    \end{stringdiagram}
    \node at (1.5,-1) {$=$};
    \begin{scope}[xshift=3cm]
      \draw[edge] (-.5,.5) -- (-.5,-2.5);
      \draw[edge] (.5,.5) -- (.5,-2.5);
      \node at (1.5,-1) {$=$};
    \end{scope}
    \begin{scope}[xshift=6cm]
      \begin{stringdiagram}{2}
        \nbraidslice{0}
        \pbraidslice{0}
      \end{stringdiagram}
    \end{scope}
  \end{tikzpicture}
  \caption{Reidemeister 2 move ($\sigma$ is an iso)}
  \label{fig:reidemeister-2}
\end{subfigure}
\begin{subfigure}{0.4\textwidth}
  \centering
  \begin{tikzpicture}[scale=.3]
    \begin{stringdiagram}{3}
      \diagslice{0}{2}{2}
      \pbraidslice{1}
      \pbraidslice{0}
      \begin{scope}[every node/.style={node distance=.75cm}]
        \node[scale=.6,above of=v0] {$\dots$};
        \node[scale=.6,below of=v0] {$\dots$};
      \end{scope}
      \node[scale=.6] at (.5,-3) {$\dots$};
    \end{stringdiagram}
    \node at (2,-1.5) {$=$};
    \begin{scope}[xshift=4cm]
      \begin{stringdiagram}{3}
        \pbraidslice{1}
        \pbraidslice{0}
        \diagslice{1}{2}{2}
        \begin{scope}[every node/.style={node distance=.75cm}]
          \node[scale=.6,above of=v2] {$\dots$};
          \node[scale=.6,below of=v2] {$\dots$};
        \end{scope}
        \node[scale=.6] at (-.5,0) {$\dots$};
      \end{stringdiagram}
    \end{scope}
  \end{tikzpicture}
  \caption{Pull-through move (naturality of $\sigma$)}
  \label{fig:pull-through}
\end{subfigure}

  \caption{Equalities satisfied by braid morphisms}
  \label{fig:braided-naturality}
\end{figure}
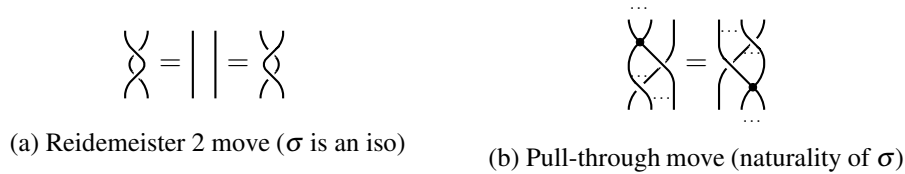

The soundness and completeness theorem of string diagrams for monoidal categories can
be extended to the case of braided monoidal
categories~\cite{joyal1991geometry}.\footnote{Soundness and completeness theorems for graphical
languages are sometimes called \emph{coherence theorems} but we avoid this terminology because of the confusion it creates with other ``coherence'' theorems.} This requires adapting the notion
of string diagram, which is now three-dimensional, and the
corresponding class of isotopies. We state the soundness and completeness theorem as
formulated by \cite{selinger2010survey}. To tell it apart from other notions of isotopy,
we use the name \emph{regular isotopy} for the class of isotopies used in this theorem.

\begin{theorem}
  \label{thm:diag-coherence-braided}
  A well-formed equation between morphisms in the language of braided
  monoidal categories follows from the axioms of braided monoidal
  categories if and only if it holds in the graphical language up to
  regular isotopy in 3 dimensions.
\end{theorem}

\begin{figure}
  \centering
\begin{tikzpicture}[scale=.3]
    \begin{stringdiagram}{4}
    \pbraidslice{1}
    \diagslice{0}{2}{1}
    \nbraidslice{0}
    \pbraidslice{1}
    \end{stringdiagram}
    \node at (3,-2) {\Large $=$};
    \begin{scope}[xshift=6cm,yshift=-.5cm]
    \begin{stringdiagram}{4}
        \nbraidslice{0}
        \diagslice{1}{2}{1}
        \pbraidslice{1}
    \end{stringdiagram}
    \end{scope}
\end{tikzpicture}
\caption{Two regularly isotopic diagrams in a braided monoidal category}
\label{fig:isotopic-braided-diagrams}
\end{figure}
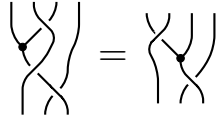

The combinatorics of string braidings have been studied extensively,
but more often from the perspective of group theory than category theory.
The braid group was introduced by~\cite{artin1947theory}.

\begin{definition}
  The braid group on $n$ strands $B_n$, is the free group generated
  by generators $\sigma_1, \dots, \sigma_{n-1}$ with equations
  \begin{align}
    \sigma_k \sigma_{k+1} \sigma_k = \sigma_{k+1} \sigma_k \sigma_{k+1} & & \text{for $1 \leq k < n-1$} \label{eqn:reidemeister-3} \\
    \sigma_i \sigma_j = \sigma_j \sigma_i & & \text{for $j - i > 1$} \label{eqn:braid-group-exchange}
  \end{align}
  for $1 \leq k < n-1$.
\end{definition}

\begin{figure}
  \centering
  \begin{subfigure}{0.4\textwidth}
  \centering
  \begin{tikzpicture}[scale=0.5]
    \begin{stringdiagram}{6}
      \pbraidslice{2}
    \end{stringdiagram}
    \node[scale=0.6] at (-2,-.5) {$\dots$};
    \node[scale=0.6] at (2,-.5) {$\dots$};
    \draw [decorate,decoration={brace,amplitude=5pt,mirror,raise=4ex}] (-2.6,-1) -- (-1.4,-1)
    node[midway,yshift=-3em,scale=.6]{$k-1$ strings};
    \draw [decorate,decoration={brace,amplitude=5pt,mirror,raise=4ex}] (1.4,-1) -- (2.6,-1)
    node[midway,yshift=-3em,scale=.6]{$n - k - 1$ strings};
  \end{tikzpicture}
  \caption{Representation of $\sigma_k$}
  \label{fig:representation-sigma-k}
\end{subfigure}
\begin{subfigure}{0.4\textwidth}
  \centering
    \begin{tikzpicture}[scale=.3]
    \begin{stringdiagram}{3}
      \pbraidslice{0}
      \pbraidslice{1}
      \pbraidslice{0}
    \end{stringdiagram}
    \node at (2,-1.5) {$=$};
    \begin{scope}[xshift=4cm]
      \begin{stringdiagram}{3}
        \pbraidslice{1}
        \pbraidslice{0}
        \pbraidslice{1}
      \end{stringdiagram}
    \end{scope}
  \end{tikzpicture}
  \caption{Representation of Equation~\ref{eqn:reidemeister-3}}
\end{subfigure}

  \caption{Graphical representation for the braid group}
  \label{fig:braid-group}
\end{figure}
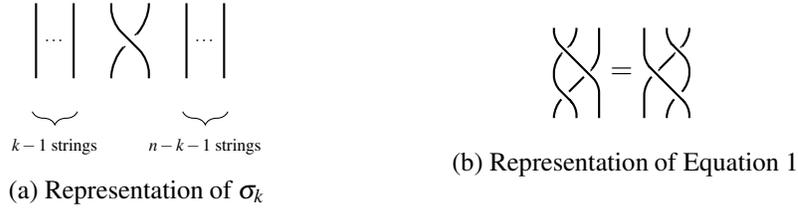

With this formalism, an element of the group represents a braid on $n$
strings. The element $\sigma_k$ represents a positive braiding of the adjacent
strings $k$ and $k+1$, and no change on other strings. Its inverse
$\sigma_k^{-1}$ is the negative braiding on the same
strings. Figure~\ref{fig:braid-group} shows how generators and
equations of the braid group can be represented graphically. Equation~\ref{eqn:reidemeister-3}
is called the Reidemeister type 3 move (or Yang-Baxter equation). One can see that it is a particular case of
the pull-through move of Figure~\ref{fig:pull-through}, where the
morphism being pulled through is a braid itself.
Equation~\ref{eqn:braid-group-exchange} corresponds to the exchange move in monoidal categories.

We can make the connection between this group-theoretic presentation
and braided monoidal categories more precise.

\begin{definition}
  A \textbf{monoidal signature} $(G,M)$ is a set of generating objects
  $G$ and a set of generating morphisms $M$. Each generating morphism
  is associated with two lists of generating objects $\dom(M)$ and $\cod(M)$,
  each being elements of $G^*$.
\end{definition}

Given a monoidal signature, one can generate the free braided monoidal
category on it. By Theorem~\ref{thm:diag-coherence-braided}, this is
the category of braided string diagrams whose vertices and edges are
labelled by generating objects and morphisms respectively. Note that
we are not imposing any additional equation between the generators:
the only equations which hold are those implied by the braided
monoidal structure itself.

\begin{proposition}[{\cite{joyal1993braided}}]
  The free braided monoidal category $\mathcal{B}$ generated by the
  signature $(\{A\}, \emptyset)$ is the braid category,
  i.e.\ $\mathcal{B}(A^n, A^n) = B_n$, the group of braids on $n$
  strands.
\end{proposition}

This can easily be seen in string diagrams: a morphism in
$\mathcal{B}$ can only be made of identities, positive and negative
braids. As a string diagram in a monoidal category, it can be drawn in
general position, where all braids appear at a different height. This
can therefore be decomposed as a sequential composition of slices
containing exactly one positive or negative braid.  The number of
wires between each slice remains constant given that braids and
identities always have as many outputs as inputs; let us call this
number $n$.  Each of these slices corresponds to a generator or
generator inverse in $B_n$.  The equations holding in
$\mathcal{B}(A^n, A^n)$ and $B_n$ are the same, hence the equality.

Therefore, braided monoidal categories generalize the braid group
by allowing for other morphisms than braids and identities. The word
problem for the braid group is well understood:

\begin{theorem}
  The word problem for the braid group $B_n$ can be solved in
  polynomial time: given two strings of generators and generator
  inverses, one can determine if they represent the same braid in
  quadratic time in the length of the strings (for a fixed $n$).
\end{theorem}

See~\cite{dehornoy2007efficient} for a review of the various techniques
which can be used to achieve this complexity. It seems hard to generalize
any of these to the case of braided monoidal categories.

Intuitively, the word problem for braided monoidal categories is
harder than the one for the braid group because of the existence of
other morphisms which can block the interaction between
braids. Because of these additional morphisms, string diagrams in
braided monoidal categories can look \emph{knotted} and the
equivalence problem for them intuitively becomes harder.
In this paper we make this intuition more precise, by showing
that the equivalence problem for braided monoidal categories
is at least as hard as the unknotting problem.

\subsection{Braided pivotal categories}

\begin{definition}
  In a monoidal category $\mathcal{C}$, an object $A \in \mathcal{C}$ has a \textbf{left adjoint} $B \in \mathcal{C}$ (or equivalently, $A$ is the \textbf{right adjoint} of $B$) when there
  are morphisms
  \begin{tikzpicture}[scale=.3,baseline=-.6cm]
    \begin{stringdiagram}{0}
      \capslice{0}
    \end{stringdiagram}
    \node[scale=.7,domainlabel] at (output_1) {$A$};
    \node[scale=.7,domainlabel] at (output_2) {$B$};
  \end{tikzpicture}
  and
  \begin{tikzpicture}[scale=.3,baseline=-.1cm]
    \begin{stringdiagram}{2}
      \cupslice{0}
    \end{stringdiagram}
    \node[scale=.7,domainlabel] at (input_1) {$B$};
    \node[scale=.7,domainlabel] at (input_2) {$A$};
  \end{tikzpicture}
  such that the \textbf{yanking equations} (or \textbf{zig-zag equations}) are satisfied:
  \[
    \vspace{-.3cm}
    \begin{tikzpicture}[scale=.3]
      \begin{scope}[xshift=12cm,yshift=-1cm]
        \begin{stringdiagram}{1}
          \capslice{1}
          \cupslice{0}
        \end{stringdiagram}
        \node[domainlabel] at (input_1) {$B$};
        \node[domainlabel] at (output_1) {$B$};
      \end{scope}
      \node at (14.25,-2.75) {$=$};
      \begin{scope}[xshift=15.50cm,yshift=-1.75cm]
        \begin{stringdiagram}{1}
          \diagslicenovertex{0}{1}{1}
          \diagslicenovertex{0}{1}{1}
        \end{stringdiagram}
        \node[domainlabel] at (input_1) {$B$};
        \node[domainlabel] at (output_1) {$B$};
      \end{scope}

      \begin{scope}[xshift=-20cm]
        \begin{scope}[xshift=19cm,yshift=-1cm]
          \begin{stringdiagram}{1}
            \capslice{0}
            \cupslice{1}
          \end{stringdiagram}
          \node[domainlabel] at (input_1) {$A$};
          \node[domainlabel] at (output_1) {$A$};
        \end{scope}
        \node at (21.25,-2.75) {$=$};
        \begin{scope}[xshift=22.50cm,yshift=-1.75cm]
          \begin{stringdiagram}{1}
            \diagslicenovertex{0}{1}{1}
            \diagslicenovertex{0}{1}{1}
          \end{stringdiagram}
          \node[domainlabel] at (input_1) {$A$};
          \node[domainlabel] at (output_1) {$A$};
        \end{scope}
      \end{scope}

    \end{tikzpicture}
    \label{eq:adjoint}
  \]
\end{definition}

\begin{definition}
  A monoidal category $\mathcal{C}$ is \textbf{left autonomous}
  if every object $A \in \mathcal{C}$ has a left adjoint $\prescript{*}{}{A}$. A monoidal category $\mathcal{C}$ is \textbf{right autonomous} if every object $A \in \mathcal{C}$ has a right adjoint $A^*$.
  A category that is both left and right autonomous is simply called \textbf{autonomous}.
\end{definition}

\begin{lemma} \label{lemma:braided-autonomous}
  Any braided monoidal category that is left autonomous is also right
  autonomous (and therefore autonomous).
\end{lemma}

\begin{proof}
  See Lemma 4.17 in \cite{selinger2010survey}.
\end{proof}

\begin{definition}
  A \textbf{strict pivotal category} is a monoidal category where every object $A$ has identical
  left and right adjoints.
\end{definition}

\noindent In a braided pivotal category, for each object $A$ there is an object $B$
with the following morphisms:
\[
\begin{tikzpicture}[scale=.3]
  \begin{stringdiagram}{0}
    \capslice{0}
  \end{stringdiagram}
  \node[scale=.7,domainlabel] at (output_1) {$A$};
  \node[scale=.7,domainlabel] at (output_2) {$B$};
  \begin{scope}[xshift=3cm]
    \begin{stringdiagram}{0}
      \capslice{0}
    \end{stringdiagram}
    \node[scale=.7,domainlabel] at (output_1) {$B$};
    \node[scale=.7,domainlabel] at (output_2) {$A$};
  \end{scope}
  \begin{scope}[xshift=6cm,yshift=-1.5cm]
    \begin{stringdiagram}{2}
      \cupslice{0}
    \end{stringdiagram}
    \node[scale=.7,domainlabel] at (input_1) {$B$};
    \node[scale=.7,domainlabel] at (input_2) {$A$};
    \begin{scope}[xshift=3cm]
      \begin{stringdiagram}{2}
        \cupslice{0}
      \end{stringdiagram}
      \node[scale=.7,domainlabel] at (input_1) {$A$};
      \node[scale=.7,domainlabel] at (input_2) {$B$};
    \end{scope}
  \end{scope}
\end{tikzpicture}
\vspace{-.5cm}
\]
such that all four yanking equations are satisfied.

\begin{definition}[{\cite{freyd1989braided}}]
  $\ROTang$ is the free braided pivotal category generated by an object represented by the symbol ``$\uparrow$''.
  We denote by ``$\downarrow$'' its adjoint.
\end{definition}

As the notations suggest, the wires of string diagrams in $\ROTang$ can be associated with an upwards or downwards orientation. We adopt the following representation for the morphisms arising from the pivotal structure:
\[
\begin{tikzpicture}[scale=.3]
  \begin{stringdiagram}{0}
    \capright{0}
  \end{stringdiagram}
  \node[scale=.7,domainlabel] at (output_1) {$\uparrow$};
  \node[scale=.7,domainlabel] at (output_2) {$\downarrow$};
  \begin{scope}[xshift=3cm]
    \begin{stringdiagram}{0}
      \capleft{0}
    \end{stringdiagram}
    \node[scale=.7,domainlabel] at (output_1) {$\downarrow$};
    \node[scale=.7,domainlabel] at (output_2) {$\uparrow$};
  \end{scope}
  \begin{scope}[xshift=6cm,yshift=-1.5cm]
    \begin{stringdiagram}{2}
      \cupright{0}
    \end{stringdiagram}
    \node[scale=.7,domainlabel] at (input_1) {$\downarrow$};
    \node[scale=.7,domainlabel] at (input_2) {$\uparrow$};
    \begin{scope}[xshift=3cm]
      \begin{stringdiagram}{2}
        \cupleft{0}
      \end{stringdiagram}
      \node[scale=.7,domainlabel] at (input_1) {$\uparrow$};
      \node[scale=.7,domainlabel] at (input_2) {$\downarrow$};
    \end{scope}
  \end{scope}
\end{tikzpicture}
\]
With this convention, we can represent any oriented knot, i.e.\ any
embedding of an oriented loop in $\mathbb{R}^3$, as a morphism
in $\ROTang$, as in Figure~\ref{fig:example-knot-rotang}.
In fact, this representation is more than a convention, as the following
theorem shows:

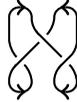
\begin{figure}
  \centering
\begin{tikzpicture}[scale=.3]
  \def\defaultroundedness{15pt}
  \begin{stringdiagram}{0}
    \startslice{4}
    \capright{0}
    \capleft{0}
    \finishslice
    \pbraidslice{1}
    \startslice{4}
    \nbraidslice{0}
    \nbraidslice{0}
    \finishslice
    \startslice{0}
    \cupright{0}
    \cupleft{0}
    \finishslice
  \end{stringdiagram}
\end{tikzpicture}
\caption{An oriented knot as a morphism in $\ROTang$}
\label{fig:example-knot-rotang}
\end{figure}

\begin{definition}[{(Definition~2.4 in \cite{freyd1989braided})}]
  A \emph{tangle} is a rectangular portion of a knot diagram. We say that two tangles are equal if there is a regular isotopy carrying one to the other in such a way that corresponding edges of the diagram are preserved setwise.
\end{definition}

\begin{theorem}[{(Theorem~3.5 in \cite{freyd1989braided})}] \label{thm:rotang}
  $\ROTang$ is the category of oriented tangles up to regular isotopy.
\end{theorem}
This means that two morphisms $f, g \in \ROTang$ are equal if and only if
their string diagrams, considered as oriented tangles in three-dimensional space
as defined above, are regularly isotopic. Regular isotopy is a more restrictive sort
of isotopy than the notion generally used in knot theory, as the following
morphisms are distinct in $\ROTang$:
\[
\begin{tikzpicture}[scale=.3]
  \begin{stringdiagram}{0}
    \capleft{0}
    \pbraidslice{0}
  \end{stringdiagram}
  \node at (2,-2) {$\neq$};
  \begin{scope}[xshift=4cm,yshift=-.5cm]
    \begin{stringdiagram}{0}
    \capright{0}
    \end{stringdiagram}
  \end{scope}
  \begin{scope}[xshift=8cm]
      \node at (-2,-2) {$\neq$};
    \begin{stringdiagram}{0}
      \capleft{0}
      \nbraidslice{0}
    \end{stringdiagram}
  \end{scope}
\end{tikzpicture}
\]
The move that equates them is called the Reidemeister type I move, which is therefore
not admissible for the string diagrams of $\ROTang$.

\begin{definition}
  A morphism $f \in \ROTang$ is an \textbf{(oriented) knot} if its string diagram has a single connected component.
\end{definition}

\subsection{The unknotting problem}

In this section we give a brief overview of the unknotting problem and some complexity results about it.

\begin{definition}
  A \textbf{topological knot} is the embedding of a loop in $\mathbb{R}^3$. Two knots $K_1$, $K_2$ are in \textbf{general isotopy} if there is an orientation-preserving homeomorphism $h$ of $\mathbb{R}^3$ such that $h(K_1) = h(K_2)$.
  A \textbf{knot diagram} is the projection of a knot on a plane, such that no two crossings happen at the same place. Additionally the diagram records the relative position of the strands at each crossing.
\end{definition}

All knots considered here will be required to be tame, i.e.\ isotopic to a polygonal knot. This gets rid of some pathological cases.

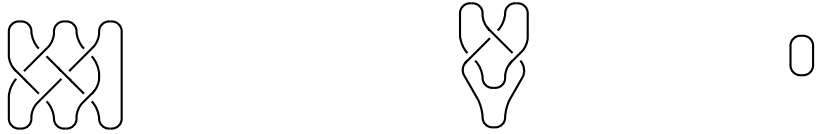
\begin{figure}
  \centering
  \begin{subfigure}{0.35\textwidth}
    \centering
    \begin{tikzpicture}[scale=0.3]
  \begin{stringdiagram}{0}
    \startslice{6}
    \capslice{0}
    \capslice{0}
    \capslice{0}
    \finishslice

    \startslice{6}
    \nbraidslice{1}
    \nbraidslice{0}
    \finishslice
    \startslice{6}
    \pbraidslice{0}
    \pbraidslice{0}
    \finishslice
    \startslice{6}
    \nbraidslice{1}
    \nbraidslice{0}
    \finishslice
    \startslice{0}
    \cupslice{0}
    \cupslice{0}
    \cupslice{0}
    \finishslice
  \end{stringdiagram}
\end{tikzpicture}
  \end{subfigure}
  \begin{subfigure}{0.35\textwidth}
    \centering
    \begin{tikzpicture}[scale=0.3]
  \begin{stringdiagram}{0}
    \startslice{4}
    \capslice{0}
    \capslice{0}
    \finishslice

    \pbraidslice{1}
    
    \startslice{4}
    \nbraidslice{0}
    \nbraidslice{0}
    \finishslice

    \cupslice{1}
    \cupslice{0}
  \end{stringdiagram}
\end{tikzpicture}
  \end{subfigure}
  \begin{subfigure}{0.15\textwidth}
    \centering
    \begin{tikzpicture}[scale=0.3]
  \begin{stringdiagram}{0}
    \capslice{0}
    \cupslice{0}
  \end{stringdiagram}
\end{tikzpicture}
  \end{subfigure}
  \caption{Some knot diagrams}
  \label{fig:example-knot}
\end{figure}

Some example knot diagrams are given in Figure~\ref{fig:example-knot}.
The Reidemeister moves are local transformations of knot diagrams which are divided in three categories, as shown in Figure~\ref{fig:reidemeister-moves}. Note that in addition to these moves, all planar isotopies are implicitly allowed, without restricting the direction of strands in any way (unlike the recumbent isotopies of monoidal string diagrams).

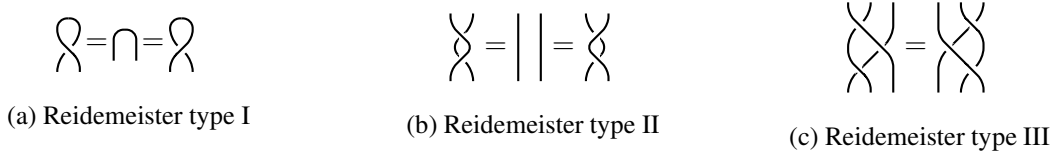
\begin{figure}
  \centering
  \begin{subfigure}{0.33\textwidth}
    \centering
    \begin{tikzpicture}[scale=.3]
  \begin{stringdiagram}{0}
    \capslice{0}
    \pbraidslice{0}
  \end{stringdiagram}
  \node at (1.25,-2) {$=$};
  \begin{scope}[xshift=2.5cm,yshift=-.5cm]
    \begin{stringdiagram}{0}
    \capslice{0}
    \end{stringdiagram}
  \end{scope}
  \begin{scope}[xshift=5cm]
      \node at (-1.25,-2) {$=$};
    \begin{stringdiagram}{0}
      \capslice{0}
      \nbraidslice{0}
    \end{stringdiagram}
  \end{scope}
\end{tikzpicture}
    \caption{Reidemeister type I}
  \end{subfigure}
  \begin{subfigure}{0.33\textwidth}
    \centering
      \begin{tikzpicture}[scale=.3]
    \begin{stringdiagram}{2}
      \pbraidslice{0}
      \nbraidslice{0}
    \end{stringdiagram}
    \node at (1.5,-1) {$=$};
    \begin{scope}[xshift=3cm]
      \draw[edge] (-.5,.5) -- (-.5,-2.5);
      \draw[edge] (.5,.5) -- (.5,-2.5);
      \node at (1.5,-1) {$=$};
    \end{scope}
    \begin{scope}[xshift=6cm]
      \begin{stringdiagram}{2}
        \nbraidslice{0}
        \pbraidslice{0}
      \end{stringdiagram}
    \end{scope}
  \end{tikzpicture}
    \caption{Reidemeister type II}
  \end{subfigure}
  \begin{subfigure}{0.3\textwidth}
    \centering
      \begin{tikzpicture}[scale=.3]
    \begin{stringdiagram}{3}
      \pbraidslice{0}
      \pbraidslice{1}
      \pbraidslice{0}
    \end{stringdiagram}
    \node at (2,-1.5) {$=$};
    \begin{scope}[xshift=4cm]
      \begin{stringdiagram}{3}
        \pbraidslice{1}
        \pbraidslice{0}
        \pbraidslice{1}
      \end{stringdiagram}
    \end{scope}
  \end{tikzpicture}
    \caption{Reidemeister type III}
  \end{subfigure}
  \caption{Reidemeister moves}
  \label{fig:reidemeister-moves}
\end{figure}

\begin{theorem}[{(Reidemeister)}]
  Two knot diagrams represent topological knots in general isotopy if and only if they are related by a sequence of Reidemeister moves.
\end{theorem}

Knot diagrams can be encoded in various ways, for instance as
four-valent planar maps where vertices are crossings and edges are parts of strands.
This makes it possible to formulate decision problems about topological knots and study their complexity.

\begin{definition}
  The unknotting problem \textsc{Unknot} is the decision problem to determine if a topological knot can be related by a general isotopy to the unknot.
  In other words, it consists in determining whether there exists a series of Reidemeister moves
  which eliminates all crossings in a given knot diagram.
\end{definition}

This problem was first formulated by~\cite{Dehn1910} and its decidability remained open until~\cite{haken1961theorie} found an algorithm for it. The problem has since attracted a lot of attention and we give a summary of the latest results about it.

\begin{theorem}[{\cite{lackenby2015polynomial}}]
  There exists a polynomial $P(c)$ such that for every knot diagram $K$ of the unknot with $c$ crossings,
  there is a sequence of Reidemeister moves unknotting it whose length
  is bounded by $P(c)$.
\end{theorem}

\begin{corollary}
  \textsc{Unknot} lies in NP.
\end{corollary}

\begin{theorem}[{\cite{lackenby2019efficient}}]
  \textsc{Unknot} lies in co-NP.
\end{theorem}

Recently, Lackenby announced a quasi-polynomial time solution to
\textsc{Unknot}, but the corresponding article has not been made
public to date. No polynomial time algorithm for this problem is known
so far.

\section{Reducing the unknotting problem to the braided pivotal word problem}

Despite the discrepancy between the general isotopy used in the unknotting problem
and the regular isotopy used in $\ROTang$, we will show that the unknotting problem
can be reduced to the word problem in $\ROTang$.
This will show that the word problem for $\ROTang$ is at least as hard
as the unknotting problem. This section is dedicated to this result.

\subsection{Writhe and turning number}

The main differences between the unknotting problem and the word
problem for $\ROTang$ is that in the latter, knots are oriented and
the Reidemeister type I move is not allowed. Because of this, we will
see in this section that we can associate a quantity called
\emph{writhe} to diagrams in $\ROTang$, which is preserved by all the
axioms of this category.

\begin{definition}
  The \textbf{writhe} (or \emph{framing number}) $W(f)$ of a diagram $f \in \ROTang$ is the sum of
  the valuations $W(b)$ for each braiding $b$ which appears in $f$:
  \begin{align*}
  W(
  \begin{tikzpicture}[scale=.3,baseline=-.2cm]
    \begin{stringdiagram}{2}
      \pbraidslice{0}
    \end{stringdiagram}
    \node[above of=output_1,node distance=0.15cm,rotate=180] {\tikz\draw[very thick,->] (0,0);};
    \node[above of=output_2,node distance=0.15cm,rotate=180] {\tikz\draw[very thick,->] (0,0);};
  \end{tikzpicture}) = +1 & &
    W(
  \begin{tikzpicture}[scale=.3,baseline=-.2cm]
    \begin{stringdiagram}{2}
      \nbraidslice{0}
    \end{stringdiagram}
    \node[above of=output_1,node distance=0.15cm,rotate=180] {\tikz\draw[very thick,->] (0,0);};
    \node[above of=output_2,node distance=0.15cm,rotate=180] {\tikz\draw[very thick,->] (0,0);};
  \end{tikzpicture}) = -1 & &
  \end{align*}
\end{definition}

\begin{definition} \label{def:turning-number}
  The \textbf{turning number} (or \emph{winding number}) $T(f)$ of a morphism $f \in \ROTang$ is the sum of the local turning numbers which appear in $f$:
  \begin{align*}
    T(
    \begin{tikzpicture}[scale=.3,baseline=-.7cm]
      \begin{stringdiagram}{0}
        \capright{0}
      \end{stringdiagram}
      \node[scale=.7,domainlabel] at (output_1) {$\uparrow$};
      \node[scale=.7,domainlabel] at (output_2) {$\downarrow$};
    \end{tikzpicture}
    ) = + 1 & &
    T(
    \begin{tikzpicture}[scale=.3,baseline=-.7cm]
     \begin{stringdiagram}{0}
      \capleft{0}
    \end{stringdiagram}
    \node[scale=.7,domainlabel] at (output_1) {$\downarrow$};
    \node[scale=.7,domainlabel] at (output_2) {$\uparrow$};
    \end{tikzpicture}
    ) = - 1 & &
    T(
    \begin{tikzpicture}[scale=.3,baseline=-.15cm]
    \begin{stringdiagram}{2}
      \cupright{0}
    \end{stringdiagram}
    \node[scale=.7,domainlabel] at (input_1) {$\downarrow$};
    \node[scale=.7,domainlabel] at (input_2) {$\uparrow$};
    \end{tikzpicture}
    ) = - 1 & &
    T(
    \begin{tikzpicture}[scale=.3,baseline=-.15cm]
      \begin{stringdiagram}{2}
        \cupleft{0}
      \end{stringdiagram}
      \node[scale=.7,domainlabel] at (input_1) {$\uparrow$};
      \node[scale=.7,domainlabel] at (input_2) {$\downarrow$};
    \end{tikzpicture}
    ) = + 1
  \end{align*}
\end{definition}
\noindent The turning number is well defined because the axioms of $\ROTang$
respect the turning number, making it independent of the particular diagram
considered.

\begin{theorem}[{\cite{trace1983reidemeister}}] \label{thm:trace}
  Let $f, g \in \ROTang$ be two knots. Then $f$ and $g$ are in general isotopy
  as knots (allowing Reidemeister type I moves) if and only if $W(f) =
  W(g)$, $T(f) = T(g)$ and there is a regular isotopy between $f$ and $g$ (disallowing
  Reidemeister type I moves).
\end{theorem}

This means that to reduce the unknot problem to the word problem for $\ROTang$, we simply need to be able to tweak diagrams to adjust their writhe and turning number without changing their general isotopy class. This is what the following section establishes.

\subsection{Unknotting in braided pivotal categories}

\begin{lemma} \label{lemma:writhe-turning-straight}
  Given a writhe $w$ and a turning number $t$ such that $2w + t$ is a
  multiple of $4$, we can construct a morphism $f \in
  \ROTang$ with domain and codomain $\uparrow$, such that $W(f) = w$ and $T(f) = t$,
  and $f$ is in general isotopy to the identity using
  Reidemeister type I moves.
\end{lemma}

\begin{proof}
  We first define the following morphisms in $\ROTang(\uparrow, \uparrow)$:
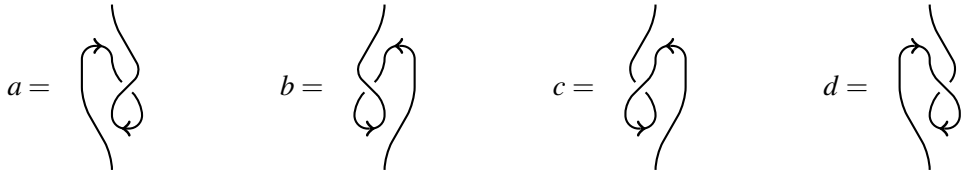
  \begin{align*}
    a =
    \begin{tikzpicture}[scale=.4,baseline=-1cm]
      \begin{stringdiagram}{1}
        \capright{0}
        \nbraidslice{1}
        \cupleft{1}
      \end{stringdiagram}
    \end{tikzpicture}
    & &
    b =
    \begin{tikzpicture}[scale=.4,baseline=-1cm]
      \begin{stringdiagram}{1}
        \capleft{1}
        \pbraidslice{0}
        \cupright{0}
      \end{stringdiagram}
    \end{tikzpicture}
  & &
  c =
  \begin{tikzpicture}[scale=.4,baseline=-1cm]
    \begin{stringdiagram}{1}
      \capleft{1}
      \nbraidslice{0}
      \cupright{0}
    \end{stringdiagram}
  \end{tikzpicture}
  & &
  d =
    \begin{tikzpicture}[scale=.4,baseline=-1cm]
    \begin{stringdiagram}{1}
      \capright{0}
      \pbraidslice{1}
      \cupleft{1}
    \end{stringdiagram}
    \end{tikzpicture}
  \end{align*}
  They have the following invariants:
  \begin{align*}
    W(a) = +1 & & W(b) = -1 & & W(c) = +1 & & W(d) = -1 \\
    T(a) = +2 & & T(b) = -2 & & T(c) = -2 & & T(d) = +2
  \end{align*}
  Let $w, t \in \mathbb{Z}$ such that $2w + t$ is a multiple of $4$.
  We construct the required morphism $f \in \ROTang(\uparrow, \uparrow)$
  by composition of $a$, $b$, $c$ and $d$ using the fact that $W(g \circ h) = W(g) + W(h)$
  and $T(g \circ h) = T(g) + T(h)$ for all $g, h \in \ROTang(\uparrow, \uparrow)$.
  Let $p = \frac{2w + t}{4}$. If $p$ is positive, we start by $p$ copies of $a$,
  otherwise $-p$ copies of $b$. Then, let $q = \frac{2w - t}{4}$. If $q$ is positive, we continue
  with $q$ copies of $c$, otherwise $-q$ copies of $d$.
  One can check that the composite has the required writhe and turning number.
\end{proof}

\begin{corollary} \label{coro:rotang-unknot}
  The general isotopy problem for knots can be reduced to the word problem for $\ROTang$.
\end{corollary}

\begin{proof}
  Given two knots $k, l$ represented as crossing diagrams in the plane,
  pick an orientation for them and turn them into morphisms
  $f, g \in \ROTang$.
  We can compute the writhe and turning number of $f$ and $g$ in polynomial time.
  
  As noted by \cite{trace1983reidemeister}, for any oriented knot $f$, $\frac{2W(f) + T(f)}{2}$ is odd.
  In other words there are $p, q \in \mathbb{Z}$ such that $2W(f) + T(f) = 4p + 2$
  and $2W(g) + T(g) = 4q + 2$. Therefore $2(W(f) - W(g)) + (T(f) - T(g)) = 4(p-q)$.
  By Lemma~\ref{lemma:writhe-turning-straight}, we can therefore construct
  a morphism $h \in \ROTang(\uparrow, \uparrow)$ such that $W(h) = W(f) - W(g)$
  and $T(h) = T(f) - T(g)$, and such that $h$ can be related by a general isotopy to a straight wire
  (so, allowing Reidemeister type I moves).

  Therefore we can insert $h$ on any strand of $g$, obtaining a morphism $g'$ which
  represents the same knot as $g$, such that $W(g) = W(f)$ and $T(g) = T(f)$. 
  By Theorem~\ref{thm:trace}, $f$ and $g'$ are in general isotopy as knots if and only if
  they are equal as morphisms of $\ROTang$. This completes the proof.
\end{proof}

\section{Reducing the unknotting problem to the braided monoidal word problem}

So far, Corollary~\ref{coro:rotang-unknot} only reduces the unknot problem to the word
problem for $\ROTang$ while our goal is to reduce it to the word problem for braided monoidal
categories. The category $\ROTang$ can be presented as a free braided monoidal category
but that requires additional equations between the generators representing the caps and cups.
In this section, we show how these equations can be eliminated too. We call \emph{unknot diagram} any knot diagram which is isotopic to the unknot.

\begin{definition}
  The category $\CC$ is the free braided monoidal category generated by objects $\{ \uparrow,
  \downarrow \}$ and morphisms $\{ \scapl, \scapr, \scupl, \scupr \}$.
\end{definition}
It is important to note that no equations are imposed between the morphism generators, unlike in
$\ROTang$. Therefore, there exists a functor from $\CC$ to $\ROTang$, mapping the generators
of $\CC$ to the corresponding units and counits in $\ROTang$, but the reverse mapping would not be
functorial.

\subsection{Cap-cup cycles}

In this section we introduce a more precise invariant than the
turning number: the sequence of caps and cups encountered while
following the strand of a knot.

\begin{definition}
  A \textbf{cap-cup cycle} is a finite sequence of elements in $\{ \scapl, \scapr, \scupl, \scupr \}$ considered
  up to cyclic permutation in which caps and cups alternate.
  The turning number of a cap-cup cycle is the sum of the turning numbers of its elements,
  defined as in Definition~\ref{def:turning-number}.
\end{definition}

The cap-cup cycle is intended to replace the turning number in a context where eliminating caps and cups using the adjunction equations is
not allowed.

\begin{definition}
  Given a knot $f \in \CC$,
  its cap-cup cycle $\ccc(f)$ can be obtained by starting from any strand in $f$, following
  it in the direction indicated by its type and recording all the caps and cups encountered
  until one travels back to the starting point.
  This cycle is invariant under all axioms of a braided monoidal category.
\end{definition}

\noindent For instance, the knot of Figure~\ref{fig:example-knot-rotang} has cap-cup cycle $(\scapl,\scupr,\scapr,\scupl)$.

\begin{lemma}
  For all knot diagrams $f \in \CC$, $\ccc(f)$ is of even length, and $T(\ccc(f)) = T(f)$.
\end{lemma}

\begin{lemma} \label{lemma:cap-cup-cycle-realization}
  For all cap-cup cycles $c$ such that $T(c) = \pm 2$, one can construct
  a knot diagram $f \in \CC$ without any crossings, such that $\ccc(f) = c$.
\end{lemma}

\begin{proof}
  By induction on the length of the cycle $c$.
  If $|c| = 2$, then $c = ( \scapr, \scupl )$ or $c = ( \scapl, \scupr )$,
  both of which can be realized by the composite of both elements.
  If $|c| > 2$, then there is an element $x \in c$ such that $T(x) = +1$
  and another element $y \in c$ with $T(y) = -1$. One can also assume that they
  are adjacent in $c$ (if all elements $x \in c$ with $T(x) = +1$ are such that the elements on their left and right also have a positive turning number, then by propagating this, all elements in the cycle have a positive turning number, which is a contradiction). Consider the cycle $c'$ obtained by removing $x$ and $y$
  from $c$. By induction, construct a knot diagram $f' \in \CC$ such that
  the cap-cup cycle of $f'$ is $c'$. Now, at the point where we removed $x$ and $y$,
  we can insert in $f'$ a zig-zag corresponding to $x$ and $y$ (in the order they
  appeared in $c$), which gives us the required knot.
\end{proof}

To generalize this lemma to knot diagrams with crossings, we introduce a new
notion of cap-cup cycle where each cap or cup can carry its own writhe.

\begin{definition}
  The set of \textbf{twisted cap-cups} is  $\mathbb{T} \coloneqq \{ \scapl, \scapr, \scupl, \scupr \} \times \mathbb{Z}$.
\end{definition}

We think of a pair $(c,w) \in \mathbb{T}$ as a cap-cup $c$ composed with braids such that the writhe of the resulting morphism is $w$. Figure~\ref{fig:twisted-cap-cups} gives a few examples of twisted cap-cups.

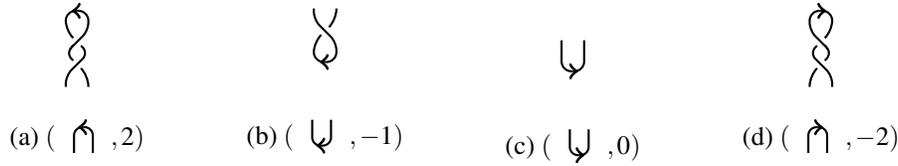
\begin{figure}
  \centering
  \begin{subfigure}{0.2\textwidth}
    \centering
    \begin{tikzpicture}[scale=0.3]
      \begin{stringdiagram}{0}
        \capleft{0}
        \nbraidslice{0}
        \nbraidslice{0}
      \end{stringdiagram}
    \end{tikzpicture}
    \caption{$(\scapl, 2)$}
  \end{subfigure}
  \begin{subfigure}{0.2\textwidth}
    \centering
    \begin{tikzpicture}[scale=0.3]
      \begin{stringdiagram}{2}
        \pbraidslice{0}
        \cupleft{0}
      \end{stringdiagram}
    \end{tikzpicture}
    \caption{$(\scupl, -1)$}
  \end{subfigure}
    \begin{subfigure}{0.2\textwidth}
      \centering
      \vspace{.55cm}
    \begin{tikzpicture}[scale=0.3]
      \begin{stringdiagram}{2}
        \cupright{0}
      \end{stringdiagram}
    \end{tikzpicture}
    \caption{$(\scupr, 0)$}
    \end{subfigure}
    \begin{subfigure}{0.2\textwidth}
    \centering
    \begin{tikzpicture}[scale=0.3]
      \begin{stringdiagram}{0}
        \capright{0}
        \pbraidslice{0}
        \pbraidslice{0}
      \end{stringdiagram}
    \end{tikzpicture}
    \caption{$(\scapr, -2)$}
  \end{subfigure}
  \caption{Examples of twisted cap-cups}
  \label{fig:twisted-cap-cups}
\end{figure}

\begin{definition}
  The \textbf{turning number} of a twisted cap-cup $(c,w) \in \mathbb{T}$ is
  defined as $T((c,w)) = (-1)^{|w|} t(c)$.
  The \textbf{writhe} of a twisted cap-cup is $W((c,w)) = w$.
  The \textbf{signature} of a twisted cap-cup is $S((c,w)) = c$ if $w$ is even,
  and $c$ with a flipped wire orientation if $w$ is odd.
\end{definition}

The signature of a twisted cap-cup is essentially obtained by applying Reidemeister type I moves
to the twisted cap-cup until no crossing remains. Therefore this preserves the domain and codomain of the morphism.
  
\begin{definition}
  A \textbf{twisted cap-cup cycle} is a finite sequence of twisted cap-cups
  up to cyclic permutation in which caps and cups alternate.
  The turning number of a cap-cup cycle is the sum of the turning numbers of its elements,
  and similarly for its writhe.
\end{definition}

\begin{definition}
  Given a twisted cap-cup cycle $c$ we define a cap-cup cycle $U(c)$ obtained by forgetting
  the writhe component in each twisted cap-cup. We also define a cap-cup cycle $S(c)$ obtained
  by taking the signature of each twisted cap-cup in the cycle.
\end{definition}

\begin{lemma} \label{lemma:twisted-cap-cup-cycle-realization}
  Let $c$ be a twisted cap-cup cycle such that $T(c) = \pm 2$. One can construct an unknot diagram $R(c) \in \CC$ such that $W(R(c)) = W(c)$ and $\ccc(R(c)) = U(c)$.
\end{lemma}

\begin{proof}
  First, notice that for all twisted cap-cup cycle $c$, $T(S(c)) = T(c)$.
  So if $T(c) = \pm 2$ then $T(S(c)) = \pm 2$ and we can apply Lemma~\ref{lemma:cap-cup-cycle-realization} to $S(c)$, obtaining a morphism $f$ such that $\ccc(f) = S(c)$.
  Now we obtain another knot diagram $R(c)$ by replacing each cap and cup of $f$ by the twisted cap-cup in $c$ it was generated from. This is possible because taking the signature of a twisted cap-cup preserves the domain and codomain of the corresponding morphism.
  We therefore obtain $W(R(c)) = W(c)$ and $\ccc(R(c)) = U(c)$ as required.
\end{proof}

\begin{lemma} \label{lemma:cap-cup-cycle-writhe-realization}
  Let $c$ be a cap-cup cycle and $w \in \mathbb{Z}$ be such that $w + \frac{T(c)}{2}$ is odd.
  Then we can construct an unknot diagram $f(c,w) \in \CC$ such that $W(f(c,w)) = w$ and $\ccc(f(c,w)) = c$.
\end{lemma}

\begin{proof}
  We can view $c$ as a twisted cap-cup cycle where all the writhe components are null.
  We will transform $c$ to incorporate the writhe $w$ in the writhe components of the cycle.

  First, consider the case where $T(c) = \pm 2$. By assumption, $w$ is therefore even.
  We can pick any element $(a,b)$ of $c$ and replace it by $(a,b+w)$, giving us a new twisted cap-cup cycle $c'$. We have $T(c') = T(c) = \pm 2$ so we can apply Lemma~\ref{lemma:twisted-cap-cup-cycle-realization}, giving the required morphism $R(c') \eqqcolon f(c,w)$.

  Second, if $T(c) = 0$. By assumption, $w$ is odd. Again, take any
  element $(a,b)$ in $c$ and replace it by $(a,b+w)$. This changes the turning number of that element,
  negating its sign. Therefore the turning number of the new twisted cap-cup cycle is $\pm 2$, and we
  are back to the previous case.

  Third, if $|T(c)| > 2$. By symmetry let us assume $T(c) > 2$. We work by induction on $T(c)$.
  There are at least two elements of $c$ with turning number $+1$, let them be $(a,b)$ and $(a',b')$.
  We replace them by $(a,b+1)$ and $(a',b'-1)$ respectively. We have $t((a,b+1)) = t((a',b'-1)) = -1$ so this reduces the turning number by $4$, keeps the writhe unchanged and keeps $U(c)$ unchanged.
  So we can obtain the required morphism by induction.
\end{proof}

The following lemma establishes that the way the writhe is spread on the elements of a twisted cap-cup cycle does not actually matter. The writhe can be transferred between any two elements without resorting to Reidemeister I or zig-zag elimination.
\begin{lemma} \label{lemma:writhe-equalizing}
  Let $c, c'$ be twisted cap-cup cycles such that $W(c) = W(c')$, $T(c) = T(c') = \pm 2$ and $U(c) = U(c')$. Then $R(c)$ is isotopic to $R(c')$ via the axioms of braided monoidal categories.
\end{lemma}

\begin{proof}
  We define a relation $\diamond$ on twisted cap-cup cycles:
  $c \diamond c'$ when $c'$ can be obtained from $c$ by replacing two consecutive elements $(a,b), (c,d)$ by $(a,b-1), (c,d+1)$.

  Let $c, c'$ be twisted cap-cup cycles as in the lemma. We first
  show that if $c \diamond c'$ then $R(c)$ is isotopic to $R(c')$
  as a braided monoidal morphism.

  If $T((a,b)) = -T((c,d))$ then the sequence $(a,b),(c,d)$ is realized in $R(c)$ as follows, up to vertical and horizontal symmetries:
  
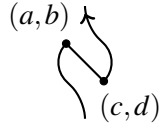
\[
  \begin{tikzpicture}[scale=0.5,every node/.style={node distance=0.5cm}]
    \begin{stringdiagram}{1}
      \diagslice{0}{0}{2}
      \diagslice{1}{2}{0}
    \end{stringdiagram}
    \node[below of=input_1,node distance=0.2cm,rotate=20] {\tikz\draw[very thick,->] (0,0);};
    \node[above left of=v0] {$(a,b)$};
    \node[below right of=v1] {$(c,d)$};
  \end{tikzpicture}
  \]
  where the morphisms are composed of a single cap or cup, followed by braids to obtain the desired writhe.
  We have the following regular isotopy:
  
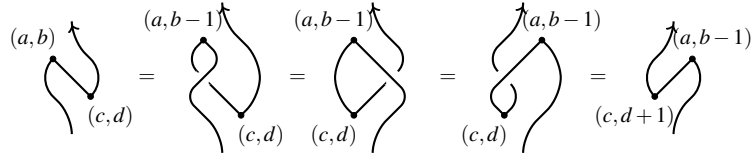
\[
  \begin{tikzpicture}[scale=0.5,every node/.style={node distance=0.5cm,scale=0.75}]
  \begin{stringdiagram}{1}
    \diagslice{0}{0}{2}
    \diagslice{1}{2}{0}
  \end{stringdiagram}
  \node[above left of=v0] {$(a,b)$};
  \node[below right of=v1] {$(c,d)$};
  \node[below of=input_1,node distance=0.2cm,rotate=20] {\tikz\draw[very thick,->] (0,0);};
  \node at (2,-1) {$=$};

  \begin{scope}[xshift=4cm,yshift=0.5cm]
    \begin{stringdiagram}{1}
      \diagslice{0}{0}{2}
      \nbraidslice{0}
      \diagslice{1}{2}{0}
    \end{stringdiagram}
    \node[above left of=v0] {$(a,b-1)$};
    \node[below right of=v2] {$(c,d)$};
    \node[below of=input_1,node distance=0.2cm,rotate=20] {\tikz\draw[very thick,->] (0,0);};
    \node at (2,-1.5) {$=$};
  \end{scope}

  \begin{scope}[xshift=8cm,yshift=0.5cm]
    \begin{stringdiagram}{1}
      \diagslice{0}{0}{2}
      \pbraidslice{1}
      \diagslice{0}{2}{0}
    \end{stringdiagram}
    \node[above left of=v0] {$(a,b-1)$};
    \node[below left of=v2] {$(c,d)$};
    \node[below of=input_1,node distance=0.2cm,rotate=20] {\tikz\draw[very thick,->] (0,0);};
    \node at (2,-1.5) {$=$};
  \end{scope}

  \begin{scope}[xshift=12cm,yshift=0.5cm]
    \begin{stringdiagram}{1}
      \diagslice{1}{0}{2}
      \nbraidslice{0}
      \diagslice{0}{2}{0}
    \end{stringdiagram}
    \node[above right of=v0] {$(a,b-1)$};
    \node[below left of=v2] {$(c,d)$};
    \node[below of=input_1,node distance=0.2cm,rotate=-20] {\tikz\draw[very thick,->] (0,0);};
    \node at (2,-1.5) {$=$};
  \end{scope}

  \begin{scope}[xshift=16cm]
    \begin{stringdiagram}{1}
      \diagslice{1}{0}{2}
      \diagslice{0}{2}{0}
    \end{stringdiagram}
    \node[above right of=v0] {$(a,b-1)$};
    \node[below left of=v1] {$(c,d+1)$};
    \node[below of=input_1,node distance=0.2cm,rotate=-20] {\tikz\draw[very thick,->] (0,0);};
  \end{scope}
  
\end{tikzpicture}
  \]
  Note that the first and last equalities are not Reidemeister I moves: they can simply be expressed as unboxing the composite morphisms $(a,b)$ and $(c,d+1)$, possibly with the help of Reidemeister II moves to create braids when required.
  This shows that $R(c)$ is isotopic to $R(c')$ as a braided monoidal morphism.

  If $T((a,b)) = T((c,d))$ then the sequence $(a,b),(c,d)$ is realized in $R(c)$ as follows, again up to vertical and horizontal symmetries:
  
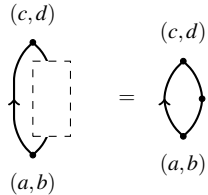
\[
  \begin{tikzpicture}[scale=0.5,every node/.style={node distance=0.5cm,scale=0.75}]
    \begin{stringdiagram}{0}
      \diagslice{0}{0}{2}
      \diagslice{1}{1}{1}
      \diagslice{1}{1}{1}
      \diagslice{0}{2}{0}
    \end{stringdiagram}
              \node at (-.5,-2) {\tikz\draw[very thick,->] (0,0);};
    \node[above of=v0] {$(c,d)$};
    \node[below of=v3] {$(a,b)$};

    \draw[dashed,fill=white] ($(v1)+(-.5,.5)$) rectangle ($(v2)+(.5,-.5)$);
    \node at (2.5,-2) {$=$};
    \begin{scope}[xshift=4cm,yshift=-.5cm]
      \begin{stringdiagram}{0}
        \diagslice{0}{0}{2}
        \diagslice{1}{1}{1}
        \diagslice{0}{2}{0}
      \end{stringdiagram}
          \node at (-.5,-1.5) {\tikz\draw[very thick,->] (0,0);};
      \node[above of=v0] {$(c,d)$};
      \node[below of=v2] {$(a,b)$};
    \end{scope}
  \end{tikzpicture}
  \]
  The dashed area in the left-hand side represents the rest of the knot. Because by construction we know that it does not cross the wire passing on its left, nor is it connected with anything else, we can abstract it away as a simple morphism taking one wire as input and one wire as output, as in the right-hand side. Then:
  
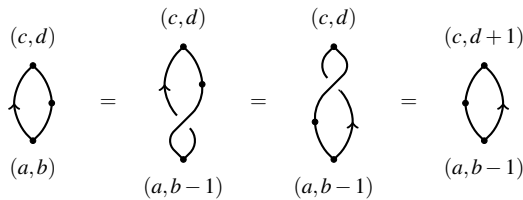
\[
\begin{tikzpicture}[scale=0.5,every node/.style={node distance=0.5cm,scale=0.75}]
  \begin{stringdiagram}{0}
    \diagslice{0}{0}{2}
    \diagslice{1}{1}{1}
    \diagslice{0}{2}{0}
  \end{stringdiagram}
  \node at (-.5,-1.5) {\tikz\draw[very thick,->] (0,0);};
  \node[above of=v0] {$(c,d)$};
  \node[below of=v2] {$(a,b)$};
  \node at (2,-1.5) {$=$};
  \begin{scope}[xshift=4cm,yshift=.5cm]
    \begin{stringdiagram}{0}
      \diagslice{0}{0}{2}
      \diagslice{1}{1}{1}
      \nbraidslice{0}
      \diagslice{0}{2}{0}
    \end{stringdiagram}
    \node at (-.5,-1.5) {\tikz\draw[very thick,->] (0,0);};
    \node[above of=v0] {$(c,d)$};
    \node[below of=v3] {$(a,b-1)$};
    \node at (2,-2) {$=$};
  \end{scope}
  \begin{scope}[xshift=8cm,yshift=.5cm]
    \begin{stringdiagram}{0}
      \diagslice{0}{0}{2}
      \nbraidslice{0}
      \diagslice{0}{1}{1}
      \diagslice{0}{2}{0}
    \end{stringdiagram}
    \node at (.5,-2.5) {\tikz\draw[very thick,->] (0,0);};
    \node[above of=v0] {$(c,d)$};
    \node[below of=v3] {$(a,b-1)$};
    \node at (2,-2) {$=$};    
  \end{scope}
  \begin{scope}[xshift=12cm]
    \begin{stringdiagram}{0}
      \diagslice{0}{0}{2}
      \diagslice{0}{1}{1}
      \diagslice{0}{2}{0}
    \end{stringdiagram}
    \node[above of=v0] {$(c,d+1)$};
    \node[below of=v2] {$(a,b-1)$};
    \node at (.5,-1.5) {\tikz\draw[very thick,->] (0,0);};
  \end{scope}
\end{tikzpicture}
  \]
  So again $R(c)$ is isotopic to $R(c')$.

  So the $\diamond$ relation respects braided monoidal isotopy.
  But now, by assumption $W(c) = W(c')$ and $U(c) = U(c')$. By a sequence of $\diamond$ steps one can transfer the writhe of any element of $c$ to any other element. So $c$ and $c'$ are related by a series of $\diamond$ steps, so they are equal as braided monoidal morphisms.
\end{proof}

\subsection{Bridge isotopy}

In this section, we introduce a notion of knot isotopy which forbids the elimination of caps and cups, but still allows Reidemeister I moves.
  
\begin{definition}
  A knot diagram $k \in \CC$ is in \textbf{bridge position} if all caps appear above of
  all cups in its string diagram. The number of caps (or equivalently cups) is called the \textbf{bridge number} of the diagram.
\end{definition}
For instance, all knot diagrams of Figure~\ref{fig:example-knot} are in bridge position. Figure~\ref{fig:bridge-position} shows a knot diagram that is not in bridge position and an equivalent diagram in bridge position.
The following lemma shows that any knot diagram can be put in bridge position without cancelling any zig-zag, as illustrated by Figure~\ref{fig:bridge-position}.

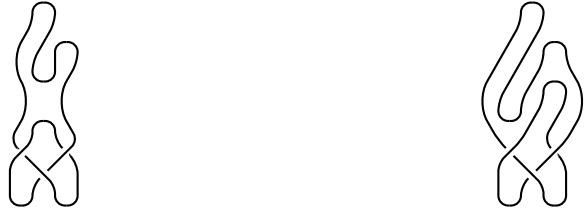
\begin{figure}
  \centering
  \begin{subfigure}{0.4\textwidth}
    \centering
    \begin{tikzpicture}[scale=0.3]
      \begin{stringdiagram}{0}
        \capslice{0}
        \capslice{2}
        \cupslice{1}
        \capslice{1}
        \startslice{4}
        \nbraidslice{0}
        \nbraidslice{0}
        \finishslice
        \pbraidslice{1}
        \startslice{0}
        \cupslice{0}
        \cupslice{0}
        \finishslice
      \end{stringdiagram}
    \end{tikzpicture}
    \caption{A knot diagram not in bridge position}
  \end{subfigure}
    \begin{subfigure}{0.4\textwidth}
    \centering
    \begin{tikzpicture}[scale=0.3]
      \begin{stringdiagram}{0}
        \capslice{0}
        \capslice{2}
        \capslice{3}
        \cupslice{1}
        \startslice{4}
        \nbraidslice{0}
        \nbraidslice{0}
        \finishslice
        \pbraidslice{1}
        \startslice{0}
        \cupslice{0}
        \cupslice{0}
        \finishslice
      \end{stringdiagram}
    \end{tikzpicture}
    \caption{An equivalent diagram in bridge position}
  \end{subfigure}
  \caption{Bridge position}
  \label{fig:bridge-position}
\end{figure}

\begin{lemma}
  Any knot diagram $k \in \CC$ can be expressed in bridge position via the axioms of braided
  monoidal categories.
\end{lemma}

\begin{proof}
  While there is a cap or cup that is not on the first or last slice of the diagram,
  pull the cup down or pull the cap up using the pull-through move (naturality of the braid).
  This move can be executed regardless of the surroundings of the cap or cup.
\end{proof}
Note that bridge positions are not unique and there are generally
multiple pull-through moves available to pull a cap or cup towards the boundary of the diagram.

\begin{definition}[{\cite{otal1982presentations,jang2019stabilization}}]
  A \textbf{bridge isotopy} between two knot diagrams in bridge position is a sequence
  of moves (including Reidemeister I) such that at each step the diagram is in bridge position.
\end{definition}

Note that because cups and caps are required to stay apart throughout the isotopy, the bridge number of the diagram is preserved by bridge isotopy.

\begin{theorem}[{\cite{otal1982presentations}}] \label{thm:otal}
  Let $K, K'$ be two diagrams of the unknot in bridge position,
  with an equal bridge number. Then they are in bridge isotopy.
\end{theorem}

\subsection{Unknotting with braided monoidal categories}

We can now combine the results above to establish a
polynomial time reduction between the unknotting problem
and the word problem for braided monoidal categories.

\begin{lemma} \label{lemma:final}
  Let $k$ be a diagram of the unknot. Then it is braided monoidal
  isotopic to $f(\ccc(k),W(k))$.
\end{lemma}

\begin{proof}
  Recall that $f(\ccc(k),W(k))$ is the diagram of the unknot constructed in Lemma~\ref{lemma:cap-cup-cycle-writhe-realization} so that its cap-cup cycle and writhe match those of $k$.

  First, the bridge number of $k$ and $f(\ccc(k), W(k))$ are equal
  since $\ccc(f(\ccc(k), W(k))) = \ccc(k)$. So by
  Theorem~\ref{thm:otal}, the two diagrams are in bridge
  isotopy. This is not quite enough for us since this bridge isotopy
  might contain Reidemeister I moves, which are not allowed in braided
  monoidal isotopy.

  To get rid of those Reidemeister I moves, we follow the same approach as Theorem~\ref{thm:trace}. First, we view all caps and cups present at all stages of the isotopy as twisted caps and cups with a null writhe component. 
  Then, scanning the isotopy from start to end, we replace Reidemeister I moves by identities (when the Reidemeister I move cancels a braiding) or by Reidemeister II moves (when the Reidemeister I move introduces a braiding). Doing so, we bundle up the leftover braid with the cap or cup in the writhe component of the twisted cap-cup.
  \[
  \begin{tikzpicture}[scale=0.3,every node/.style={node distance=0.25cm,scale=0.5}]
  \begin{stringdiagram}{0}
    \capleft{0}
    \nbraidslice{0}
  \end{stringdiagram}
  \node at (1.5,-2) {$\rightarrow$};
  \begin{scope}[xshift=3cm,yshift=-.5cm]
    \begin{stringdiagram}{0}
      \capleft{0}
    \end{stringdiagram}
  \end{scope}
  \node[scale=1.5] at (7.5,-2) {becomes};
  \begin{scope}[xshift=11cm,yshift=-.75cm]
    \begin{stringdiagram}{0}
      \diagslice{0}{2}{2}
      \nbraidslice{0}
    \end{stringdiagram}
    \node[above of=v0] {$(a,b)$};
    \node at (2.25,-1.25) {$\rightarrow$};
    \begin{scope}[xshift=3cm,yshift=-.5cm]
      \begin{stringdiagram}{0}
        \diagslice{0}{2}{2}
      \end{stringdiagram}
      \node[above of=v0] {$(a,b+1)$};
    \end{scope}
  \end{scope}
\end{tikzpicture}
  \vspace{-.4cm}
  \]
  Since the isotopy is a bridge isotopy, caps and cups never get cancelled so adding this writhe component does not prevent any further step of the isotopy.

  After this transformation, the target of the isotopy might have some
  additional writhe components on some caps and cups. But the original
  target was $f(\ccc(k),W(k))$, which was defined as $R(c)$, the
  realization of a twisted cap-cup cycle $c$.  So the new target can
  also be seen as the realization of another twisted cap-cup cycle
  $c'$, which has identical writhe and cap-cup cycle, because it is in
  braided monoidal isotopy with the source.  Therefore we can apply
  Lemma~\ref{lemma:writhe-equalizing} and obtain a braided monoidal
  isotopy between the new target of our isotopy and $f(\ccc(k),W(k))$,
  completing the proof.
\end{proof}

\begin{theorem}
  The unknotting problem can be polynomially reduced to the word problem for braided monoidal categories.
\end{theorem}

\begin{proof}
  Given a knot diagram $k$, we convert it to a braided monoidal word problem as follows. First, we orient it in an arbitrary way, obtaining a morphism $k' \in \ROTang$. We compute its writhe $W(k')$ and cap-cup cycle $\ccc(k')$. Then we compute $f(\ccc(k'), W(k'))$.
  All these steps can be done in polynomial time. Finally, the
  corresponding word problem is to determine whether $k$ and $f(\ccc(k'), W(k'))$ are in braided monoidal isotopy. If they are, then $k$ is the unknot. If they are not then by Lemma~\ref{lemma:final}, $k$ is knotted.
\end{proof}

A summary of the process of deciding whether a knot is unknotted given an algorithm to solve the word problem for braided monoidal categories is given in Figure~\ref{fig:whole-process}.

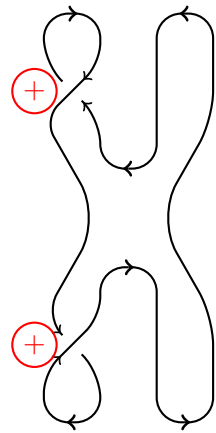
\begin{figure}
  \centering
  \begin{subfigure}{0.45\textwidth}
    \centering
    \begin{tikzpicture}[scale=0.75]
      \begin{stringdiagram}{0}
        \startslice{6}
        \capslice{0}
        \capslice{0}
        \capslice{0}
        \finishslice

        \startslice{6}
        \nbraidslice{1}
        \nbraidslice{0}
        \finishslice
        \startslice{6}
        \pbraidslice{0}
        \pbraidslice{0}
        \finishslice
        \startslice{6}
        \nbraidslice{1}
        \nbraidslice{0}
        \finishslice
        \wlog{Starting slice of cups}
        \startslice{0}
        \cupslice{0}
        \cupslice{0}
        \cupslice{0}
        \finishslice
      \end{stringdiagram}
    \end{tikzpicture}
    \caption{Initial knot diagram}
  \end{subfigure}
  \begin{subfigure}{0.45\textwidth}
    \centering
    \begin{tikzpicture}[scale=0.75]
      \begin{stringdiagram}{0}
        \startslice{6}
        \capright{0}
        \capright{0}
        \capleft{0}
        \finishslice

        \startslice{6}
        \nbraidslice{1}
        \nbraidslice{0}
        \finishslice
        \startslice{6}
        \pbraidslice{0}
        \pbraidslice{0}
        \finishslice
        \startslice{6}
        \nbraidslice{1}
        \nbraidslice{0}
        \finishslice
        \wlog{Starting slice of cups}
        \startslice{0}
        \cupleft{0}
        \cupleft{0}
        \cupright{0}
        \finishslice
        
      \end{stringdiagram}
    \end{tikzpicture}
    \caption{Oriented knot as a morphism in $\ROTang$}
  \end{subfigure}
    \begin{subfigure}{0.45\textwidth}
    \centering
    \begin{tikzpicture}[scale=0.75]
      \begin{stringdiagram}{0}
        \startslice{6}
        \capright{0}
        \capright{0}
        \capleft{0}
        \finishslice

        \startslice{6}
        \nbraidslice{1}
        \nbraidslice{0}
        \finishslice
        \startslice{6}
        \pbraidslice{0}
        \pbraidslice{0}
        \finishslice
        \startslice{6}
        \nbraidslice{1}
        \nbraidslice{0}
        \finishslice
        \wlog{Starting slice of cups}
        \startslice{0}
        \cupleft{0}
        \cupleft{0}
        \cupright{0}
        \finishslice
        
      \end{stringdiagram}
        \foreach \x/\y in {1/2,1/4,2/3,3/4} {
  \node[rotate=225]  at ($(place_\x_\y)+(.3,-.3)$) {\tikz\draw[thick,->] (0,0);};
  }
    \foreach \x/\y in {2/2,3/1} {
  \node[rotate=315]  at ($(place_\x_\y)+(.3,.3)$) {\tikz\draw[thick,->] (0,0);};
    }

 \foreach \x/\y in {1/5,2/4,3/3,3/5} {
  \node[rotate=135]  at ($(place_\x_\y)+(-.3,-.3)$) {\tikz\draw[thick,->] (0,0);};
 }

 \foreach \x/\y in {3/2,4/3} {
  \node[rotate=45]  at ($(place_\x_\y)+(-.3,.3)$) {\tikz\draw[thick,->] (0,0);};
  }
      \foreach \x/\y in {1/1,2/1,2/2,3/1} {
      \node[node distance=.5cm,left of=v\x-\y,thick,inner sep=2pt,red,circle,draw] {$+$};
    }
    \foreach \x/\y in {1/2,3/2} {
      \node[node distance=.5cm,left of=v\x-\y,thick,inner sep=2pt,blue,circle,draw] {$-$};
    }

    \end{tikzpicture}
    \caption{Computing the writhe: $+2$ in this case}
    \end{subfigure}
    \begin{subfigure}{0.45\textwidth}
      $( \scapl, \scupl, \scapr, \scupl, \scapr, \scupr )$
      \caption{The cap-cup cycle}
    \end{subfigure}
    \begin{subfigure}{0.45\textwidth}
      $( (\scapl, 0), (\scupl, 0), (\scapr, 1),$
      $ (\scupl, 1), (\scapr, 0), (\scupr,0) )$
      \caption{The twisted cap-cup cycle of Lemma~\ref{lemma:cap-cup-cycle-writhe-realization}}
    \end{subfigure}
    \begin{subfigure}{0.45\textwidth}
      \centering
      \begin{tikzpicture}[scale=0.75]
         \begin{stringdiagram}{0}
     \startslice{4}
     \capright{0}
     \capleft{0}
     \finishslice
     \nbraidslice{0}
     \cupleft{1}
     \capright{1}
     \nbraidslice{0}
     \startslice{0}
     \cupleft{0}
     \cupright{0}
     \finishslice
   \end{stringdiagram}
   \node[rotate=135]  at ($(place_2_2)+(-.3,.7)$) {\tikz\draw[thick,->] (0,0);};
   \node[rotate=45]  at ($(place_2_2)+(-.3,.3)$) {\tikz\draw[thick,->] (0,0);};
   \node[rotate=225]  at ($(place_4_1)+(.3,-.3)$) {\tikz\draw[thick,->] (0,0);};
   \node[rotate=315]  at ($(place_5_1)+(.3,.3)$) {\tikz\draw[thick,->] (0,0);};

   \foreach \x/\y in {1/1,4/1} {
      \node[node distance=.5cm,left of=v\x-\y,thick,inner sep=2pt,red,circle,draw] {$+$};
    }

    \end{tikzpicture}
      \caption{The realization of Lemma~\ref{lemma:twisted-cap-cup-cycle-realization}}
    \end{subfigure}
    \caption{Entire process of detecting knotedness using a solution to the braided monoidal word problem}
    \label{fig:whole-process}
\end{figure}


\begin{corollary}
  The word problem for the 3-cells of free Gray categories is at least as hard as the unknotting problem.
\end{corollary}

\begin{proof}
  Implied by the characterization of doubly degenerate Gray categories as braided monoidal categories~\citep{gurski2011periodic}.
\end{proof}

\section*{Conclusion}

We have established a connection between two areas. On one side, word problems arising naturally in category theory, which have not been studied much from a computational perspective so far. On the other side, the unknotting problem, which has been studied by knot theorists for more than a century.

Our hope with this connection is to make it evident that much more work is required on word problems in category theory, especially if we hope to develop practical proof assistants for higher categories. To our knowledge, no algorithm for the braided monoidal word problem is known to date.

\begin{conjecture}
  The word problem for 3-cells of Gray categories (and hence cells of braided monoidal categories) is decidable.
\end{conjecture}

Again, the word problem we mean here is deciding the equality of morphisms up to the axioms of Gray categories and nothing else. Note that the naive algorithm consisting in exploring all expressions reachable from a given expression does not terminate, since the Reidemeister II move can be applied indefinitely. This makes approaches such as that of \cite{makkai2005word} inapplicable. Furthermore it is possible that the word problem becomes undecidable at a higher level, perhaps for similar reasons that the isotopy of four-dimensional manifolds is undecidable~\cite{markov1958insolubility,boone1968recursively}.

Another natural question arising from our work is whether the problem of knot equivalence could be reduced to the word problem for braided monoidal categories. Knot equivalence is the problem of determining if two knot diagrams represent the same knot. In this context, it seems more difficult to suppress the need for the yanking equations, so our results do not seem to adapt easily to this more general case.

\bibliographystyle{eptcs}
\bibliography{references}

\begin{thebibliography}{10}
\providecommand{\bibitemdeclare}[2]{}
\providecommand{\surnamestart}{}
\providecommand{\surnameend}{}
\providecommand{\urlprefix}{Available at }
\providecommand{\url}[1]{\texttt{#1}}
\providecommand{\href}[2]{\texttt{#2}}
\providecommand{\urlalt}[2]{\href{#1}{#2}}
\providecommand{\doi}[1]{doi:\urlalt{http://dx.doi.org/#1}{#1}}
\providecommand{\eprint}[1]{ArXiv:\urlalt{https://arxiv.org/abs/#1}{#1}}
\providecommand{\bibinfo}[2]{#2}

\bibitemdeclare{article}{artin1947theory}
\bibitem{artin1947theory}
\bibinfo{author}{Emil \surnamestart Artin\surnameend} (\bibinfo{year}{1947}):
  \emph{\bibinfo{title}{Theory of Braids}}.
\newblock {\sl \bibinfo{journal}{Ann. of Math}}
  \bibinfo{volume}{48}(\bibinfo{number}{2}), pp. \bibinfo{pages}{101--126},
  \doi{10.2307/1969218}.

\bibitemdeclare{incollection}{boone1968recursively}
\bibitem{boone1968recursively}
\bibinfo{author}{W.~W. \surnamestart Boone\surnameend},
  \bibinfo{author}{W.~\surnamestart Haken\surnameend} \&
  \bibinfo{author}{V.~\surnamestart Po{\'e}naru\surnameend}
  (\bibinfo{year}{1968}): \emph{\bibinfo{title}{On {{Recursively Unsolvable
  Problems}} in {{Topology}} and {{Their Classification}}}}.
\newblock In \bibinfo{editor}{H.~Arnold \surnamestart Schmidt\surnameend},
  \bibinfo{editor}{K.~\surnamestart Sch{\"u}tte\surnameend} \&
  \bibinfo{editor}{H.~J. \surnamestart Thiele\surnameend}, editors: {\sl
  \bibinfo{booktitle}{Studies in {{Logic}} and the {{Foundations}} of
  {{Mathematics}}}}, {\sl \bibinfo{series}{Contributions to {{Mathematical
  Logic}}}}~\bibinfo{volume}{50}, \bibinfo{publisher}{{Elsevier}}, pp.
  \bibinfo{pages}{37--74}, \doi{10.1016/S0049-237X(08)70518-4}.

\bibitemdeclare{article}{Dehn1910}
\bibitem{Dehn1910}
\bibinfo{author}{M.~\surnamestart Dehn\surnameend} (\bibinfo{year}{1910}):
  \emph{\bibinfo{title}{\"Uber Die Topologie Des Dreidimensionalen Raumes. (Mit
  16 Figuren Im Text)}}.
\newblock {\sl \bibinfo{journal}{Mathematische Annalen}} \bibinfo{volume}{69},
  pp. \bibinfo{pages}{137--168}, \doi{10.1007/BF01455155}.

\bibitemdeclare{article}{dehornoy2007efficient}
\bibitem{dehornoy2007efficient}
\bibinfo{author}{Patrick \surnamestart Dehornoy\surnameend}
  (\bibinfo{year}{2007}): \emph{\bibinfo{title}{Efficient Solutions to the
  Braid Isotopy Problem}}.
\newblock {\sl \bibinfo{journal}{arXiv:math/0703666}},
  \doi{10.1016/j.dam.2007.12.009}.
\newblock \eprint{math/0703666}.

\bibitemdeclare{article}{delpeuch2020word}
\bibitem{delpeuch2020word}
\bibinfo{author}{Antonin \surnamestart Delpeuch\surnameend}
  (\bibinfo{year}{2020}): \emph{\bibinfo{title}{The Word Problem for Double
  Categories}}.
\newblock {\sl \bibinfo{journal}{Theory and Applications of Categories}}
  \bibinfo{volume}{35}, pp. \bibinfo{pages}{1--18}.
\newblock \eprint{1907.09927}.

\bibitemdeclare{article}{delpeuch2018normalization-1}
\bibitem{delpeuch2018normalization-1}
\bibinfo{author}{Antonin \surnamestart Delpeuch\surnameend} \&
  \bibinfo{author}{Jamie \surnamestart Vicary\surnameend}
  (\bibinfo{year}{2018}): \emph{\bibinfo{title}{Normalization for Planar String
  Diagrams and a Quadratic Equivalence Algorithm}}.
\newblock {\sl \bibinfo{journal}{to appear in Logical Methods in Computer
  Science}}.
\newblock \eprint{1804.07832}.

\bibitemdeclare{article}{freyd1989braided}
\bibitem{freyd1989braided}
\bibinfo{author}{Peter~J \surnamestart Freyd\surnameend} \&
  \bibinfo{author}{David~N \surnamestart Yetter\surnameend}
  (\bibinfo{year}{1989}): \emph{\bibinfo{title}{Braided Compact Closed
  Categories with Applications to Low Dimensional Topology}}.
\newblock {\sl \bibinfo{journal}{Advances in Mathematics}}
  \bibinfo{volume}{77}(\bibinfo{number}{2}), pp. \bibinfo{pages}{156--182},
  \doi{10.1016/0001-8708(89)90018-2}.

\bibitemdeclare{article}{gurski2011periodic}
\bibitem{gurski2011periodic}
\bibinfo{author}{Nick \surnamestart Gurski\surnameend} \&
  \bibinfo{author}{Eugenia \surnamestart Cheng\surnameend}
  (\bibinfo{year}{2011}): \emph{\bibinfo{title}{The Periodic Table of
  N-Categories {{II}}: Degenerate Tricategories}}.
\newblock {\sl \bibinfo{journal}{Cahiers de Topologie et G\'eom\'etrie
  Diff\'erentielle Cat\'egoriques}} \bibinfo{volume}{52}(\bibinfo{number}{2}),
  p.~\bibinfo{pages}{45}.

\bibitemdeclare{article}{haken1961theorie}
\bibitem{haken1961theorie}
\bibinfo{author}{W.~\surnamestart Haken\surnameend} (\bibinfo{year}{1961}):
  \emph{\bibinfo{title}{Theorie Der {{Normalfl\"achen}}}}.
\newblock {\sl \bibinfo{journal}{Acta Math.}} \bibinfo{volume}{105}, pp.
  \bibinfo{pages}{245--375}, \doi{10.1007/BF02559591}.

\bibitemdeclare{article}{jang2019stabilization}
\bibitem{jang2019stabilization}
\bibinfo{author}{Yeonhee \surnamestart Jang\surnameend},
  \bibinfo{author}{Tsuyoshi \surnamestart Kobayashi\surnameend},
  \bibinfo{author}{Makoto \surnamestart Ozawa\surnameend} \&
  \bibinfo{author}{Kazuto \surnamestart Takao\surnameend}
  (\bibinfo{year}{2019}): \emph{\bibinfo{title}{Stabilization of Bridge
  Decompositions of Knots and Bridge Positions of Knot Types ({{The}} Theory of
  Transformation Groups and Its Applications)}}.
\newblock {\sl \bibinfo{journal}{RIMS Kokyuroku}} \bibinfo{volume}{2135}, pp.
  \bibinfo{pages}{23--28}.

\bibitemdeclare{article}{joyal1986braided}
\bibitem{joyal1986braided}
\bibinfo{author}{Andr{\'e} \surnamestart Joyal\surnameend} \&
  \bibinfo{author}{Ross \surnamestart Street\surnameend}
  (\bibinfo{year}{1986}): \emph{\bibinfo{title}{Braided Monoidal Categories}}.
\newblock {\sl \bibinfo{journal}{Mathematics Reports}} \bibinfo{volume}{86008}.

\bibitemdeclare{article}{joyal1991geometry}
\bibitem{joyal1991geometry}
\bibinfo{author}{Andr{\'e} \surnamestart Joyal\surnameend} \&
  \bibinfo{author}{Ross \surnamestart Street\surnameend}
  (\bibinfo{year}{1991}): \emph{\bibinfo{title}{The Geometry of Tensor
  Calculus, {{I}}}}.
\newblock {\sl \bibinfo{journal}{Advances in Mathematics}}
  \bibinfo{volume}{88}(\bibinfo{number}{1}), pp. \bibinfo{pages}{55--112},
  \doi{10.1016/0001-8708(91)90003-p}.

\bibitemdeclare{article}{joyal1993braided}
\bibitem{joyal1993braided}
\bibinfo{author}{Andr{\'e} \surnamestart Joyal\surnameend} \&
  \bibinfo{author}{Ross \surnamestart Street\surnameend}
  (\bibinfo{year}{1993}): \emph{\bibinfo{title}{Braided Tensor Categories}}.
\newblock {\sl \bibinfo{journal}{Advances in Mathematics}}
  \bibinfo{volume}{102}(\bibinfo{number}{1}), pp. \bibinfo{pages}{20--78},
  \doi{10.1006/aima.1993.1055}.

\bibitemdeclare{article}{lackenby2015polynomial}
\bibitem{lackenby2015polynomial}
\bibinfo{author}{Marc \surnamestart Lackenby\surnameend}
  (\bibinfo{year}{2015}): \emph{\bibinfo{title}{A Polynomial Upper Bound on
  {{Reidemeister}} Moves}}.
\newblock {\sl \bibinfo{journal}{Annals of Mathematics}}, pp.
  \bibinfo{pages}{491--564}, \doi{10.4007/annals.2015.182.2.3}.

\bibitemdeclare{article}{lackenby2019efficient}
\bibitem{lackenby2019efficient}
\bibinfo{author}{Marc \surnamestart Lackenby\surnameend}
  (\bibinfo{year}{2019}): \emph{\bibinfo{title}{The Efficient Certification of
  Knottedness and {{Thurston}} Norm}}.
\newblock {\sl \bibinfo{journal}{arXiv:1604.00290 [math]}}.
\newblock \eprint{1604.00290}.

\bibitemdeclare{unpublished}{makkai2005word}
\bibitem{makkai2005word}
\bibinfo{author}{M.~\surnamestart Makkai\surnameend} (\bibinfo{year}{2005}):
  \emph{\bibinfo{title}{The Word Problem for Computads}}.

\bibitemdeclare{article}{markov1958insolubility}
\bibitem{markov1958insolubility}
\bibinfo{author}{A.~\surnamestart Markov\surnameend} (\bibinfo{year}{1958}):
  \emph{\bibinfo{title}{{The insolubility of the problem of homeomorphy}}}.
\newblock {\sl \bibinfo{journal}{Doklady Akademii Nauk SSSR}}
  \bibinfo{volume}{121}, pp. \bibinfo{pages}{218--220}.

\bibitemdeclare{article}{otal1982presentations}
\bibitem{otal1982presentations}
\bibinfo{author}{Jean-Pierre \surnamestart Otal\surnameend}
  (\bibinfo{year}{1982}): \emph{\bibinfo{title}{{Pr\'esentations en ponts du
  n\oe ud trivial}}}.
\newblock {\sl \bibinfo{journal}{C. R. Acad. Sci., Paris, S\'er. I}}
  \bibinfo{volume}{294}, pp. \bibinfo{pages}{553--556}.

\bibitemdeclare{incollection}{selinger2010survey}
\bibitem{selinger2010survey}
\bibinfo{author}{P.~\surnamestart Selinger\surnameend} (\bibinfo{year}{2010}):
  \emph{\bibinfo{title}{A {{Survey}} of {{Graphical Languages}} for {{Monoidal
  Categories}}}}.
\newblock In \bibinfo{editor}{Bob \surnamestart Coecke\surnameend}, editor:
  {\sl \bibinfo{booktitle}{New {{Structures}} for {{Physics}}}}, {\sl
  \bibinfo{series}{Lecture {{Notes}} in {{Physics}}}} \bibinfo{volume}{813},
  \bibinfo{publisher}{{Springer Berlin Heidelberg}}, pp.
  \bibinfo{pages}{289--355}, \doi{10.1007/978-3-642-12821-9\_4}.

\bibitemdeclare{article}{trace1983reidemeister}
\bibitem{trace1983reidemeister}
\bibinfo{author}{Bruce \surnamestart Trace\surnameend} (\bibinfo{year}{1983}):
  \emph{\bibinfo{title}{On the {{Reidemeister}} Moves of a Classical Knot}}.
\newblock {\sl \bibinfo{journal}{Proceedings of the American Mathematical
  Society}} \bibinfo{volume}{89}(\bibinfo{number}{4}), pp.
  \bibinfo{pages}{722--724}, \doi{10.1090/S0002-9939-1983-0719004-4}.

\end{thebibliography}

\appendix

\end{document}